\documentclass[a4paper]{article}
\usepackage[margin=3.3cm]{geometry}

\usepackage[T1]{fontenc}
\usepackage[english]{babel}

\usepackage{amsmath}
\usepackage{amssymb}
\usepackage{amsthm}

\usepackage{bbm}
\usepackage{mathrsfs}

\usepackage[disable]{todonotes}

\usepackage{titlesec}
\titleformat*{\subsection}{\bfseries}

\usepackage{mathtools}
\mathtoolsset{showonlyrefs}

\usepackage{enumitem}
\setenumerate{label=(\roman*)}

\usepackage{csquotes}

\usepackage{hyperref}

\usepackage{tikz-cd}

\newtheorem{lemma}{Lemma}[section]
\newtheorem{theorem}[lemma]{Theorem}
\newtheorem{proposition}[lemma]{Proposition}
\newtheorem{corollary}[lemma]{Corollary}
\theoremstyle{definition}
\newtheorem{definition}[lemma]{Definition}
\newtheorem{example}[lemma]{Example}
\newtheorem{remark}[lemma]{Remark}

\newcommand{\integers}{\mathbb{Z}}
\newcommand{\reals}{\mathbb{R}}

\newcommand{\naturals}{\mathbb{N}}

\newcommand{\gL}{\mathrm{L}}
\DeclareMathOperator*{\essup}{\mathrm{ess-sup}}

\newcommand{\gC}{\mathrm{C}}

\newcommand{\Id}{\mathrm{Id}}
\newcommand{\clos}{\mathrm{clos}}

\newcommand{\skalarprodukt}[2]{\left\langle #1 \, {,} \, #2 \right\rangle}
\newcommand{\bild}{\text{Im}}
\newcommand{\kernn}{\text{Ker}}

\DeclareMathOperator{\spann}{\mathrm{span}}

\newcommand{\gd}{\mathrm{d}}

\newcommand{\dreistrichnorm}{|\kern-0.25ex|\kern-0.25ex|}

\newcommand{\labelirgendwo}[1]{\phantomsection\label{#1}}
\newcommand{\refirgendwo}[2]{\hyperref[#1]{#2}}

\newcommand{\review}[1]{#1}

\definecolor{greenyellow}{HTML}{8CC53C}
\setuptodonotes{color=orange}

\newcommand{\kommentar}[1]{\todo[inline,color=greenyellow
	]{#1}}
\newcommand{\anmerkung}[1]{\todo[color=greenyellow,size=\tiny]{#1}}

\newcommand{\ada}[1]{\emph{(#1).}} 

\author{
	Michael Bitzer and Ingo Steinwart\\
	University of Stuttgart\\
	Faculty 8: Mathematics and Physics\\
	Institute for Stochastics and Applications\\
	D-70569 Stuttgart Germany \\
	\texttt{\small michael.bitzer@mathematik.uni-stuttgart.de}\\
	\texttt{\small ingo.steinwart@mathematik.uni-stuttgart.de}
} 

\title{Spectral representations of interpolation spaces of reproducing kernel Hilbert spaces}

\begin{document}
	\frenchspacing
	\maketitle

	\begin{abstract}
		In statistical learning theory, interpolation spaces of the form $[\gL^2,H]_{\theta,r}$, where $H$ is a reproducing kernel Hilbert space, are in widespread use. So far, however, they are only well understood for fine index $r=2$. We generalise existing results from $r=2$ to \review{$1\leq r \leq \infty$}. In particular, we present a spectral decomposition of such spaces, analyse their embedding properties, and describe connections to the theory of Banach spaces of functions. 
		Additionally, we present example applications of our results to regularisation error estimation in statistical learning. 
	\end{abstract}
	
	
	
	\section{Introduction}

\todo[inline,color=cyan]{Es gibt ein paar neue Absätze und neue Literatur in der Einleitung. Fehlt noch wichtige Literatur? }
Since their introduction, reproducing kernel Hilbert spaces (RKHS) have proven to be fundamental objects in various branches of mathematical theory and in applications. 
They possess many desirable properties, as they combine the geometry of a Hilbert space with the continuous evaluation functionals of a function space. 
Apart from applications in functional analysis and numerics, statistical learning theory takes advantage of this situation by considering a large variety of kernel-based algorithms. 

Recent years showed that in the statistical analysis of such learning algorithms, certain interpolation spaces play a key role. To be more precise, for an RKHS $H$, spaces of the form  
\begin{align} \label{formula interpolation space fine index 2}
	[\gL^2(\nu), H]_{\theta,2} 
\end{align}
are important for the analysis of kernel-based least squares regression, see e.g. \cite{sobolev-learning-rates}, for kernel interpolation, see e.g. \cite{tizian_superconvergence}, for estimations of the conditional mean embedding, see e.g. \cite{regularized-conditional-mean-embedding}, and for the analysis of learning curves, see e.g. \cite{learning_curves}. 
Similarly, these spaces occur naturally for certain questions related to Gaussian processes \review{\cite{KLE,karvonen,Gauss_disintegration,GP_paths_RKHS}}, which recently have drawn interest due to their connection to neural networks \cite{NTK}. Thereby, the associated kernels are strongly related to Sobolev spaces \cite{Max_NTK}, which are an important application of this paper. 

The main reason for choosing $2$ as the fine index in \eqref{formula interpolation space fine index 2} is the spectral representation of these interpolation spaces. More precisely, the spaces \eqref{formula interpolation space fine index 2} are given as the range of powers of the integral operator associated to the RKHS \cite[Section 4]{general_mercer}. Using eigenvalues and eigenfunctions of this operator hence yields many insights into interpolation spaces and is very helpful in statistical learning theory \review{\cite{svm,GP_RKHS_survey}}. 


Aside from these practical advantages of fine index $2$, considering other values turned out as beneficial. In particular, the following two cases are of special interest.  
\begin{itemize}
	\item In certain boundary cases, interpolation spaces of fine index $1$ enable embeddings into $\gL^\infty(\nu)$ that are not possible for other fine indexes. For examples and applications of such embeddings, see e.g. \cite[p. 4]{optimal-rates-least-squares}. 
	\item Interpolation spaces of fine index $\infty$ allow for a precise analysis of approximation errors, see e.g. \cite[Theorem 3.1]{approximation_error_interpolation_spaces}. 
\end{itemize}
For these reasons, we consider real interpolation spaces of the general form 
\begin{align} \label{formula interpolation space}
	[\gL^2(\nu), H]_{\theta,r} \, \text{,}
\end{align}
that is, \review{we allow for all fine indexes $1 \leq r \leq \infty$}. 
Our aim is to present methods and results that treat \review{these interpolation spaces for all such fine indexes $r$} in a joint fashion. The results we will derive address the following main questions: 
\begin{itemize}
	\item Do the spaces \eqref{formula interpolation space} and their norms still possess a spectral representation? 
	\item The interpolation spaces consist of equivalence classes of functions. When can we identify the spaces \eqref{formula interpolation space} with Banach spaces consisting of functions? 
	\item When are the spaces \eqref{formula interpolation space} continuously embedded into $\gL^\infty(\nu)$? 
\end{itemize}
The first question is clearly motivated by the practicality of spectral representations, as indicated by the papers mentioned above. 
Positive answers to the second question have applications in all situations where a point evaluation of functions is necessary. A main example is given by the theory of Gaussian processes in the context of machine learning, as paths cannot be equivalence classes but need to be defined everywhere, see e.g. \cite[Assumption 1]{Gaussian_processes_physics}, or by the analysis of function spaces generated by neural networks, see e.g. \cite{RKBS-neural-networks}. 
The third question will turn out to be tightly connected to the second one, but applications go beyond that. In particular, the statistical analysis of kernel-based learning algorithms relies on such embeddings, see e.g. \review{\cite[Theorem 7.23]{svm},\cite[Section 3]{sobolev-learning-rates},\cite[p. 4]{optimal-rates-least-squares},\cite[Section 4]{vector_valued_learning_rates}}. \todo{Notation Zitate?}
Our results on the second and third question improve earlier findings such as \review{\cite[Section 5]{KLE},\cite[Section 5]{general_mercer}} even in the Hilbert space case $r=2$. \todo{Notation Zitate?}


This work is organised as follows: we start by introducing the necessary notation and fundamental definitions in Section \ref{section preliminaries}. Then the main results are provided in Section \ref{section results}. In Section \ref{section application RERM}, we present an example application to error estimation in statistical learning theory. All proofs and the ideas behind the results will be given in Section \ref{section proofs}. Finally, some auxiliary results can be found in the three appendices. 
%
%
	
	\section{Preliminaries} \label{section preliminaries}

Given normed $\reals$-vector spaces $E_1$ and $E_2$, we use $E_1 \simeq E_2$ to denote isometric isomorphism, $E_1 \triangleq E_2$ for norm equivalent spaces and $E_1 \hookrightarrow E_2$ if $E_1 \subset E_2$ and the embedding is continuous.

For a measurable space $(X, \mathcal{A})$, we denote the set of measurable functions $X \rightarrow \reals$ by $\mathcal{L}^{0} (X)$. 
The set of bounded functions in $\mathcal{L}^{0} (X)$ is denoted by $\mathcal{L}^{\infty} (X)$, and this space is equipped with the supremum norm. 

Let us assume that additionally, a measure $\nu$ is given on $X$. Then for $1 \leq p < \infty$, we consider the space $\mathcal{L}^{p} (\nu)$ of all functions in $\mathcal{L}^{0} (X)$ with finite seminorm $\Vert f \Vert_{\mathcal{L}^p(\nu)} := (\int_X |f|^p \gd \nu)^\frac{1}{p}$. 
Similarly, $\mathcal{L}^{\infty}(\nu)$ denotes the space of essentially bounded functions in $\mathcal{L}^0 (X)$ equipped with the seminorm $\Vert f \Vert_{\mathcal{L}^{\infty}(\nu)} := \essup_{x \in X} |f(x)|$. \anmerkung{Beachte Unterscheidung zwischen $\mathcal{L}^{\infty}(\nu)$ und $\mathcal{L}^{\infty} (X)$}
Furthermore, we consider equivalence classes $[f]_{\sim}$ of functions that are $\nu$-almost everywhere equal to $f \in \mathcal{L}^{0} (X)$. 
For $E \subset \mathcal{L}^{0} (X)$, we denote the set of all equivalence classes of functions in $E$ as $[E]_{\sim}$. 
In particular, for $1 \leq p \leq \infty$, $\gL^p(\nu) := [\mathcal{L}^{p} (\nu)]_{\sim}$ is equipped with the classical $\gL^p$-norm inherited from $\mathcal{L}^{p} (\nu)$. 
These spaces are Banach spaces. 
For such indices $p$, we extend the typical arithmetic by $\frac{1}{\infty}:=0$. 
If $\nu$ is fixed and if $w: X \mapsto [0, \infty)$ is a measurable \enquote{weight} function, we denote the $\gL^p$-space with respect to the measure $w d\nu$ by $\gL^p(w)$. For the counting measure $\nu$, the same applies to $\ell^p(w)$. 



\subsection{Interpolation spaces} \labelirgendwo{subsection interpolation spaces}
Let us introduce some background from the theory of interpolation spaces of the real interpolation method. For more details, see \cite[Chapter 5]{bennet_sharpley} or \cite[Chapter 3]{bergh_loefstroem}. 
In the following, we only consider interpolation spaces of Banach spaces $E_1$ and $E_2$ where $E_2 \hookrightarrow E_1$. 

To define the real interpolation spaces, we use the so-called $K$-method. To this end, for all $x \in E_1$ and $t>0$, we define the $K$-functional as 
\begin{align} \label{formula K-functional}
	K(x,t,E_1,E_2) := 
	\inf_{ y \in E_2 } \left( \Vert x-y \Vert_{E_1} + t \Vert y \Vert_{E_2} \right) \text{.}
\end{align}\anmerkung{Äquivalenz zur allgemeinen Definition siehe \cite[nach Prop 4.5]{general_mercer}}
For $x \in E_1$ and the parameters $0<\theta<1$ and $1\leq r < \infty$, this enables us to set  
\begin{align}
	\Vert x \Vert_{[E_1,E_2]_{\theta,r}} 
	:= \left( \int_0^{\infty} t^{-\theta r -1} K^r(x,t, E_1, E_2) \gd t \right)^{\frac{1}{r}}
\end{align}
and similarly, for $r = \infty$, we let  
\begin{align}
	\Vert x \Vert_{[E_1,E_2]_{\theta,\infty}} := \sup_{t>0} \left( t^{-\theta} K(x,t, E_1, E_2) \right) \text{.}
\end{align}
Now the interpolation space of $E_1$ and $E_2$ to the parameters $0<\theta<1$ and $1\leq r \leq \infty$ is defined as  
\begin{align}
	[E_1,E_2]_{\theta,r} := \{ x \in E_1 \, : \, \Vert x \Vert_{[E_1,E_2]_{\theta,r}} < \infty\} 
\end{align}
and this space is equipped with the $\Vert \, \cdot \, \Vert_{[E_1,E_2]_{\theta,r}}$-norm. 
These spaces are again Banach spaces, see \cite[Proposition 1.8 in Chapter 5]{bennet_sharpley}. 
In fact, such interpolation spaces perform a kind of interpolation of Banach spaces, as they are \enquote{in between} $E_1$ and $E_2$ in the sense that $E_2 \hookrightarrow [E_1,E_2]_{\theta,r} \hookrightarrow E_1$, see Lemma \ref{lemma interpolation space subsets}. 
\review{This definition allows for $0<r<1$ as well, and these values yield new spaces \cite{quasi_Banach_interpolation_spaces}. However, as the resulting spaces are only quasi-Banach spaces, we restrict to $1\leq r \leq \infty$. }

Note that the definition of real interpolation spaces presented here is just one of several alternatives. Other ways, such as the $J$-method, yield equivalent norms. 
In certain situations, such equivalent norms can be helpful to gain insight into interpolation spaces. 

\subsection{Reproducing kernel Hilbert spaces and Banach spaces of functions} \labelirgendwo{subsection RKHS}
Let $X$ be a set and suppose that $E$ is a Banach space that consists of functions $f: X \rightarrow \reals$\anmerkung{mit Vektorraumstruktur von Funktionen}. If for all $x \in X$, the evaluation maps 
\begin{align}
	\delta_x: E \rightarrow \reals, \quad f \mapsto f(x)
\end{align}
are bounded functionals on $E$, then $E$ is called a Banach space of functions (BSF), see  \cite{RKBS} for more details. 

If additionally, the norm of such a Banach space of functions $H$ is induced by an inner product, for each $x \in X$, we can find a function $k(\cdot,x) \in H$ that is associated to $\delta_x$ by the Riesz representation theorem. Then $H$ is a reproducing kernel Hilbert space (RKHS) on $X$ with kernel $k$. We refer to \cite[Chapter 4]{svm} for fundamental results on RKHS. 

Now let $(X, \mathcal{A}, \mu)$ be a measure space. 
Suppose that $H$ is an RKHS of a measurable and bounded kernel $k$ on $X$ such that $H$ is compactly embedded into $\gL^2(\nu)$, that is the \enquote{embedding} $I_k: H \rightarrow \gL^2(\nu)$ given by $f \mapsto[f]_{\sim}$ is a compact operator. 
In this case, we consider the self-adjoint, positive and compact operator $T_{k} := I_{k} I_{k}^*$, called the integral operator of $k$, as it is given by 
\begin{align}
	T_k: \gL^2(\nu) \rightarrow \gL^2(\nu), \quad f \mapsto \left[ \int_X k(\cdot,y) f(y) \gd \nu (y) \right]_{\sim} \text{,}
\end{align}
see e.g. \cite[Lemma 2.2]{general_mercer}. 
Let us denote the non-zero eigenvalues of $T_k$ in non-increasing order by $\mu=(\mu_i)_{i \in I}$, where $I = \{1,...,n\}$ or $I = \naturals$. 
Due to the positivity and compactness of $T_k$, we know that $\mu_i > 0$ and, in case that $I$ is infinite, $\lim_{i \rightarrow \infty} \mu_i = 0$. 

As proven in \cite[Lemma 2.12]{general_mercer},\anmerkung{Dass die $e_i$ messbar sind steht da nur implizit} in such a situation, there are functions $e_i \in H$ such that  $([e_i]_{\sim})_{i \in I} \subset \gL^2(\nu)$ forms an ONS in $\gL^2(\nu)$ consisting of eigenvectors of the integral operator $T_k$, such that additionally, $(\sqrt{\mu_i} e_i)_{i \in I}$ forms an ONS inside $H$. 

The range of $I_k$, that is $[H]_{\sim}$, becomes a Hilbert space by using the quotient norm $\Vert [f]_{\sim} \Vert := \inf \{ \Vert g \Vert_H \, | \, I_k g = [f]_{\sim}  \}$. 

\todo[inline,color=cyan]{Ab hier kleinere Veränderungen: }

We summarise the above in the following assumption: 
\vspace{2mm} \newline  \noindent \textbf{Assumption H. } \labelirgendwo{assumption RKHS}
	Suppose that $(X, \mathcal{A}, \nu)$ is a measure space and that $H$ is an RKHS of a measurable kernel $k$ on $X$ such that $H$ is compactly embedded into $\gL^2(\nu)$. Then, for a suitable index set $I=\{1,...,n\}$ or $I=\naturals$, we denote the eigenvalues of $T_k:\gL^2(\nu) \rightarrow \gL^2(\nu)$ by $\mu_i$ and the associated eigenfunctions by $e_i$ for $i \in I$. 
	
	\vspace{2mm}  \noindent \textbf{Assumption H+. } \labelirgendwo{assumption RKHS strong}
	Let Assumption \refirgendwo{assumption RKHS}{H} be satisfied. In addition, let $\nu$ be $\sigma$-finite. Furthermore, assume that $k$ is a bounded kernel.  
\vspace{2mm}

Let us recall two concepts from topology and measure theory. Firstly, let $X$ be a set and suppose that $M$ is a set of functions $X \rightarrow \reals$. Then the coarsest topology $\tau(M)$ on $X$ such that all functions in $M$ are continuous is called the initial topology on $X$ with respect to $M$. 
Secondly, let $(X,\mathcal{A})$ be a measurable space and $(X,\mathcal{T})$ a topological space. Then a measure $\nu$ on $(X,\mathcal{A})$ is called $\mathcal{T}$-positive, if  $\mathcal{T} \subset \mathcal{A}$, i.e. each open set is measurable, and if for all non-empty open sets $A$, it holds that $\nu(A)>0$. 

\vspace{2mm} \noindent \textbf{Assumption T. } \labelirgendwo{assumption T}
Let $(X, \mathcal{A}, \nu)$ be a measure space such that $\gL^2(\nu)$ is separable. 
Moreover, let $H$ be an RKHS over $X$ of a measurable kernel $k$ such that Assumption \refirgendwo{assumption RKHS}{H} is satisfied and $\nu$ is $\tau(H)$-positive. We use $I$, $e_i$ and $\mu_i$ as in Assumption \refirgendwo{assumption RKHS}{H}.

\begin{remark}
	We give a sufficient condition that allows to verify Assumption \refirgendwo{assumption T}{T} in the most common situations. 
	To this end, suppose a topology $\mathcal{T}$ on $X$ is given such that 
	\begin{enumerate}
		\item $\mathcal{T} \subset \mathcal{A}$, i.e. $\mathcal{A}$ contains the Borel-$\sigma$-algebra of $\mathcal{T}$, 
		\item $\nu$ is $\mathcal{T}$-positive, and 
		\item $k$ is $\mathcal{T} \otimes \mathcal{T}$-continuous. 
	\end{enumerate}
	Then, due to (iii), all functions in $H$ are $\mathcal{T}$-continuous and hence $\tau(H) \subset \mathcal{T}$. By (ii), it follows that all non-empty open sets in $\tau(H)$ have positive measure, so $\nu$ is $\tau(H)$-positive as well. Hence $\tau(H) \subset \mathcal{A}$ is $\nu$-positive. 
	Note however that the separability of $\gL^2(\nu)$ has to be checked separately. 
	
	In particular, assume that $X$ is a subset of $\reals^n$ that satisfies $\mathrm{clos}(\mathrm{int}(X)) = \mathrm{clos}(X)$, where $\mathrm{clos}$ and $\mathrm{int}$ denote the closure and the interior, respectively.\anmerkung{Beweis: (i), (iii), Separabilität ok, (ii): Sei $\lambda$ Lebesgue-Maß, sei $\{\}\neq A\subset \reals^n$ offen. Dann ist $\nu(A\cap X) = \lambda(A \cap X) > 0$ da $A \cap \mathrm{int}(X) \neq \{\}$ offen ist. Es ist $\neq \{\}$, da sich offene Menge $A$ mit $\clos X=\clos \mathrm{int} X$ schneidet $\Leftrightarrow$ $A$ schneidet sich mit $\mathrm{int} X$ nach Def. von $\clos$} Then Assumption \refirgendwo{assumption T}{T} is satisfied if we use the Borel-$\sigma$-algebra and the Lebesgue measure with a continuous kernel on $X$. 
	More generally, we may replace $\reals^n$ by a second-countable locally compact group equipped with the Borel-$\sigma$-algebra and the Haar measure. \anmerkung{zweitabzählbar für Separabilität, siehe mathoverflow 42310, allgemeiner: Cohn Prop 3.4.5}
\end{remark}

\todo[inline,color=cyan]{Bis hier kleinere Veränderungen }

\subsection{Power spaces} \label{subsection power spaces}
Our goal is to describe the interpolation spaces from \eqref{formula interpolation space} together with their norms by a spectral representation, i.e. in terms of eigenvalues and eigenfunctions of the operator $T_k$. 
To this end, for a \enquote{weight sequence} $\mu:I \rightarrow (0,\infty)$, we consider the following partition of $I$: 
\begin{align} \label{formula def Mj}
	M_j:= \left\{i \in I \, \middle| \, 2^j < \mu_i \leq 2^{j+1} \right\}, \quad j \in \integers \text{.}
\end{align}


\begin{definition} \label{def ell spaces}
	Let $(\mu_i)_{i \in I}$ be a non-increasing sequence of positive numbers and $b = (b_i)_{i \in I}$ be a sequence of real numbers. Then for $\theta>0$\anmerkung{Hier wird zwar nur $\theta \leq 1 $ gebraucht, aber man kann an sich auch höhere $\theta$ betrachten wie in \cite[Def. 4.1]{general_mercer}} and $1\leq r \leq \infty$, let 
	\begin{align} \label{formula power space norm inner part}
		c_j(b) := \left( \sum_{i \in M_j} b_i^2 \mu_i^{-\theta} \right)^{\frac{1}{2}}
	\end{align}
	and 
	\begin{align} \label{formula power space norm}
		\Vert b \Vert_{(\mu,\theta,r)} 
		:= \Vert (c_j(b))_{j \in \integers} \Vert_{\ell^r} \text{.}
	\end{align}
	Furthermore, let $\ell^{(\mu,\theta,r)} := \left\{ b \, : \, \Vert b \Vert_{(\mu,\theta,r)} < \infty \right\}$ be equipped with the $\Vert \, \cdot \, \Vert_{(\mu,\theta,r)}$-norm. 
\end{definition}
It is not hard to see that, as the $\mu_i$ are positive, $\Vert \cdot \Vert_{(\mu,\theta,r)}$ is in fact a norm on $\ell^{(\mu,\theta,r)}$. The choice of powers of $2$ in \eqref{formula def Mj} is for notational convenience. Replacing it by any $s>1$ yields equivalent norms.\anmerkung{Siehe alten Beweis mit $s$} It will turn out\anmerkung{\refirgendwo{proof ell spaces Banach}{later}} that $\ell^{(\mu,\theta,r)}$ is actually a Banach space, see Remark \ref{remark isometric isomorphisms}. 
Note that in the case $r=2$, one finds that $\ell^{(\mu,\theta,r)} = \ell^2 (\mu^{-\theta})$ with equal norms. In consequence, the following definition generalises the notion of a power of an RKHS from \cite[Section 4]{general_mercer}. 
\review{Definition \ref{def ell spaces} generalises to $0<r<1$. As the resulting spaces are quasi-Banach spaces, we restrict to $1\leq r \leq \infty$. }

\begin{definition} \label{def power space}
	Let $H$, $e_i$ and $\mu_i$ be as introduced in Assumption \refirgendwo{assumption RKHS}{H}. 
	Then the power space of $H$ to the parameters $0<\theta<1$ and $1 \leq r \leq \infty$ is defined as 
	\begin{align}
		[H]_\sim^{(\theta,r)} 
		:= \left\{ f \in \gL^2(\nu) \, \middle\vert \, \exists b \in \ell^{(\mu,\theta,r)} \text{ with } f = \sum_{i \in I} b_i [e_i]_{\sim} \right\}
	\end{align}
	with norm given by 
	\begin{align}
		\left\Vert \sum_{i \in I} b_i [e_i]_{\sim} \right\Vert_{[H]_\sim^{(\theta,r)}} := \Vert b \Vert_{(\mu, \theta, r)} \text{.}
	\end{align}
\end{definition}

Note that this norm is well-defined due to the uniqueness of the Fourier coefficients $b_i$. 
Recall that in \cite[Section 4]{general_mercer}, only the spaces $[H]_\sim^{(\theta,2)}$ were introduced and investigated. Moreover, for $r \neq 2$, the spaces $[H]_\sim^{(\theta,r)}$ are no longer Hilbert spaces. 


	\section{Results} \label{section results}

\todo[inline,color=cyan]{Ab hier kleinere Veränderungen und neue Bemerkungen: }

Our main results are provided in the following. The first result states the norm equivalence of the power spaces $[H]_\sim^{\theta,r}$ to the interpolation spaces of the form \eqref{formula interpolation space}, as previously mentioned. 
This is the foundation for the subsequent two results, which show that we can view these interpolation spaces as BSFs, in case they are embedded into $\gL^\infty(\nu)$. 
The last result then analyses in which situations such an embedding exists. 
\begin{theorem}\label{thm equivalent norm}
	Let $0 < \theta < 1$, $1 \leq r \leq \infty$, and suppose that $H$ is as in Assumption \refirgendwo{assumption RKHS}{H}. Then 
	\begin{align}
		[H]_\sim^{(\theta,r)} = [\gL^2(\nu), [H]_{\sim}]_{(\theta,r)}
	\end{align}
	and the norms on these spaces are equivalent. 
\end{theorem} 

Note that as $[H]_\sim$ and $ \gL^2(\nu)$ consist of equivalence classes of functions, so does the interpolation space considered in Theorem \ref{thm equivalent norm}. Not only for applications in learning theory, it is of interest whether such spaces can be viewed as actual function spaces with continuous evaluations functionals by a suitable choice of representatives. 
Note that for $r\neq 2$, the spaces $[\gL^2(\nu), [H]_{\sim}]_{(\theta,r)}$ are no Hilbert spaces any more, so we have to expect BSFs instead of RKHSs. 

\begin{theorem} \label{thm interp space BSF}
	Let $(X,\mathcal{A},\nu)$ and $H$ be as in Assumption \refirgendwo{assumption RKHS strong}{H+}\anmerkung{$\sigma$-endlich, da Lifting verwendet wird}. 
	Moreover, let $0<\theta<1$ and $1\leq r \leq \infty$ be such that the embedding $[H]_\sim^{(\theta,r)} \hookrightarrow \gL^\infty (\nu)$ exists and is continuous. 
	Then there is a BSF $\overline{H}_{\nu}^{(\theta,r)}$ on $X$ such that the operator 
	\begin{align}
		\overline{H}_{\nu}^{(\theta,r)} \rightarrow [H]_\sim^{(\theta,r)}, \quad f \mapsto [f]_\sim
	\end{align}
	is a well-defined isometric isomorphism with 
	\begin{align}
		\Vert \overline{H}_{\nu}^{(\theta,r)} \hookrightarrow \mathcal{L}^\infty (X) \Vert 
		= \Vert [H]_\sim^{(\theta,r)} \hookrightarrow \gL^\infty (\nu) \Vert \text{.}
	\end{align}
	Additionally, there is a $\nu$-null set $N \subset X$ such that for all $f\in \overline{H}_{\nu}^{(\theta,r)}$, there is a $(b_i)_{i \in I} \in \ell^{(\mu,\theta,r)}$ such that $f(x)=\sum_{i \in I} b_i e_i(x)$ for all $x \in X \setminus N$. 
\end{theorem}
Theorem \ref{thm interp space BSF} ensures the existence of a BSF, whose functions may differ from the actual functions in $H$ on sets of measure zero. 
Under Assumption \refirgendwo{assumption RKHS strong}{T}, we can make the BSF more explicit. Then we can take orthogonal series of eigenfunctions as elements of the BSF, which finally enables to embed $H$ into this BSF, as the following Theorem states. 
\begin{theorem} \label{thm interp space BSF with topology}
	Let $(X,\mathcal{A},\nu)$, $H$, $\mu_i$, and $e_i$ be as in Assumption \refirgendwo{assumption RKHS strong}{T}\anmerkung{$\sigma$-endlich und $\nu$-vollständig ohne lifting nicht mehr nötig}. 
	Moreover, let $0<\theta<1$ and $1\leq r \leq \infty$ be such that the embedding $[H]_\sim^{(\theta,r)} \hookrightarrow \gL^\infty (\nu)$ exists and is continuous. Then the following statements are true: 
	\begin{enumerate}
		\item The series $\sum_{i \in I} b_i e_i $ converges uniformly for all $b \in \ell^{(\mu,\theta,r)}$. 
		\item The space 
		\begin{align} \label{formula BSF assumption T}
			{H}_{\nu}^{(\theta,r)}
			:= \left\{ \sum_{i \in I} b_i e_i \, \middle| \, b \in \ell^{(\mu,\theta,r)} \right\}
		\end{align}
		is a BSF, where the BSF-norm is inherited from the $\ell^{(\mu,\theta,r)}$-norm of the Fourier coefficients $b_i$. 
		\item The mapping 
		\begin{align}
			{H}_{\nu}^{(\theta,r)} \rightarrow [H]_\sim^{(\theta,r)}, \quad 
			\sum_{i \in I} b_i e_i \mapsto \sum_{i \in I} b_i [e_i]_\sim
		\end{align}
		is an isometric isomorphism. 
		\item It holds that  
		\begin{align}
			\Vert {H}_{\nu}^{(\theta,r)} \hookrightarrow \gC_b (X, \tau(H)) \Vert 
			= \Vert [H]_\sim^{(\theta,r)} \hookrightarrow \gL^\infty (\nu) \Vert \text{.}
		\end{align}
		\item We have $H \hookrightarrow {H}_{\nu}^{(\theta,r)}$. 
	\end{enumerate}		
\end{theorem}

The existence of an embedding of the power space $[H]_\sim^{(\theta,r)}$ into $\gL^\infty(\nu)$ seems to be a crucial requirement of Theorems \ref{thm interp space BSF} and \ref{thm interp space BSF with topology}. 
Let us elaborate in which situations we can expect such an embedding. 
For fine index $r=2$, this question has been approached in \cite[Section 5]{general_mercer}. We extend this approach to $1 \leq r \leq \infty$. 


\begin{theorem} \label{thm embedding equivalence}
	Let $(X,\mathcal{A},\nu)$, $H$, $\mu_i$, and $e_i$ be as in Assumption \refirgendwo{assumption RKHS}{H}. \anmerkung{$\sigma$-endlich da Lifting verwendet wird}
	Moreover, let $\kappa>0$, $0 < \theta < 1$, and $1\leq r \leq \infty$, where $\frac{1}{r} + \frac{1}{r'} = 1$. 
	Consider the following statements: 
	\begin{enumerate}\item It holds that 
		\begin{align} \label{formula embedding equivalent condition sup}
			\sup_{x \in X} \left\Vert \left( e_i(x) \mu_i^{\theta} \right)_{i \in I}\right\Vert_{(\mu,\theta,r')} 
			\leq \kappa \text{.}
		\end{align} 
		
		\item It holds that 
		\begin{align} \label{formula embedding equivalent condition}
			\essup_{x \in X} \left\Vert \left( e_i(x) \mu_i^{\theta} \right)_{i \in I}\right\Vert_{(\mu,\theta,r')} 
			\leq \kappa \text{.}
		\end{align} 
		
		\item The embedding $[H]_\sim^{(\theta,r)} \hookrightarrow \gL^\infty (\nu)$ exists and is continuous with operator norm at most $\kappa$. 
	\end{enumerate}
	Then we always have (i) $\Rightarrow$ (ii) $\Rightarrow$ (iii). 
	Moreover, under Assumption \refirgendwo{assumption RKHS strong}{H+}, we have (i) $\Rightarrow$ (ii) $\Leftrightarrow$ (iii), 
	and under Assumption \refirgendwo{assumption T}{T}, we even have (i) $\Leftrightarrow$ (ii) $\Leftrightarrow$ (iii). 
\end{theorem}

\begin{remark}
	Suppose that Assumption \refirgendwo{assumption RKHS}{H} is satisfied. Recall that the power kernel associated to $k$ is defined by its Mercer representation as 
	\begin{align}
		k_\nu^\theta (x,x') = \sum_{i \in I} \mu_i^\theta e_i(x) e_i(x')
	\end{align}
	for $x,x' \in X$, see \cite[Proposition 4.2]{general_mercer}. Using this notion, we can reformulate the terms in \eqref{formula embedding equivalent condition sup} and \eqref{formula embedding equivalent condition} as follows. 
	\begin{align}
		\left\Vert \left( e_i(x) \mu_i^{\theta} \right)_{i \in I}\right\Vert_{(\mu,\theta,r')}
		= \left\Vert \sum_{i \in I}  e_i(x) \mu_i^{\theta} [e_i]_{\sim} \right\Vert_{[H]_\sim^{(\theta,r')}} 
		&= \left\Vert\left[ \sum_{i \in I}  e_i(x) \mu_i^{\theta} e_i \right]_{\sim} \right\Vert_{[H]_\sim^{(\theta,r')}} \\
		&= \left\Vert\left[ k_{\nu}^{\theta}(x,\cdot) \right]_{\sim} \right\Vert_{[H]_\sim^{(\theta,r')}} \text{.}
	\end{align}
	Moreover, if $\sum_{i \in I} b_i e_i$ converges pointwise everywhere for all $b \in \ell^{(\mu,\theta,r')}$, we may define $H_\nu^{(\theta,r')}$ as in \eqref{formula BSF assumption T} and remove all $[\,\cdot\, ]_\sim$ to consider functions instead of equivalence classes. Under Assumption \refirgendwo{assumption T}{T}, this is possible due to Theorem \ref{thm interp space BSF with topology}.  
\end{remark}

\begin{remark}
	Suppose there is a $C>0$ such that the eigenfunctions $e_i$ satisfy $\Vert e_i \Vert_{\gL^\infty(\nu)} < C$ for all $i \in I$. 
	Then, by a countable union of null sets, we find that $|e_i(x)|<C$ for all $i \in I$ and for $\nu$-almost all $x \in X$. 
	Then 
	\begin{align}
		\essup_{x \in X} \left\Vert \left( e_i(x) \mu_i^{\theta} \right)_{i \in I}\right\Vert_{(\mu,\theta,r')} 
		\leq C \Vert (\mu_i^\theta)_{i \in I} \Vert_{(\mu,\theta,r')} \text{,}
	\end{align} 
	and hence condition \eqref{formula embedding equivalent condition} is satisfied for some $\kappa>0$ whenever $\Vert (\mu_i^\theta)_{i \in I} \Vert_{(\mu,\theta,r')} < \infty$. 
	
	More explicitly, the term of interest $\Vert (\mu_i^\theta)_{i \in I} \Vert_{(\mu,\theta,r')}$ is given by 
	\begin{align}
		\left(\sum_{j \in \integers} \left( \sum_{i \in M_j} \mu_i^{\theta} \right)^{\frac{r'}{2}} \right)^{\frac{1}{r'}}
		\quad \text{if } 1\leq r'< \infty \quad \text{ and by } \quad
		\sup_{j \in \integers} \left( \sum_{i \in M_j} \mu_i^{\theta} \right)^{\frac{1}{2}}
		\quad \text{if } r'=\infty \text{.}
	\end{align}
\end{remark}

\begin{remark}
	We illustrate the statements in Theorem \ref{thm embedding equivalence} in a concrete situation. To this end, suppose that $\mu_i=i^{-\alpha}$ for $\alpha>1$, which is linked to Sobolev spaces \cite{Sobolev_RKHS}\anmerkung{damit $\theta=\frac{1}{\alpha} \in (0,1)$}. 
	Then we find for $j \in \integers$ that 
	\begin{align}
		M_j = \{ i \in \naturals \, | \, 2^j < i^{-\alpha} \leq 2^{j+1}\}
		= \{ i \in \naturals \, | \, 2^{-\frac{j+1}{\alpha}} \leq i < 2^{-\frac{j}{\alpha}}\}
		= [2^{-\frac{j+1}{\alpha}}, 2^{-\frac{j}{\alpha}}) \cap \naturals
		\text{,}
	\end{align}
	so $M_j=\{\}$ for $j\geq 0$. 
	We hence note that $\min M_j \geq 2^{-\frac{j+1}{\alpha}}$ and $\max M_j \leq 2^{-\frac{j}{\alpha}}$. Furthermore, we have 
	\begin{align}
		 2^{-\frac{j}{\alpha}} - 2^{-\frac{j+1}{\alpha}} - 1 
		 \leq |M_j| 
		 \leq 2^{-\frac{j}{\alpha}} - 2^{-\frac{j+1}{\alpha}} + 1
	\end{align}
	By that, we estimate for $j<0$
	\begin{align}
		\sum_{i \in M_j} i^{-1} 
		\leq |M_j| 2^{\frac{j+1}{\alpha}} 
		\leq (2^{-\frac{j}{\alpha}} - 2^{-\frac{j+1}{\alpha}} + 1) 2^{\frac{j+1}{\alpha}} 
		\leq 2^{\frac{1}{\alpha}} - 1 + 2^{\frac{j+1}{\alpha}}
		\leq 1+2^{\frac{j+1}{\alpha}}
		\leq 1+2^{\frac{1}{\alpha}} \text{.}
	\end{align}
	Hence for $\theta= \frac{1}{\alpha}$ with $r=1$ and $r'=\infty$, we get 
	\begin{align}
		\Vert (\mu_i^\theta)_{i \in I} \Vert_{(\mu,\theta,r')}
		= \sup_{j \in \integers} \left( \sum_{i \in M_j} \mu_i^{\theta} \right)^{\frac{1}{2}}
		= \sup_{j \in \integers} \left( \sum_{i \in M_j} i^{-1} \right)^{\frac{1}{2}}
		\leq \left( 1+2^{\frac{1}{\alpha}} \right)^{\frac{1}{2}}
		< \infty \text{.}
	\end{align}
	However, for $r=2$ with $r'=2$, we get 
	\begin{align}
		\Vert (\mu_i^\theta)_{i \in I} \Vert_{(\mu,\theta,r')}
		= \left(\sum_{j \in \integers} \left( \sum_{i \in M_j} \mu_i^{\theta} \right)^{\frac{r'}{2}} \right)^{\frac{1}{r'}}
		= \left(\sum_{j \in \integers}  \sum_{i \in M_j} i^{-1}  \right)^{\frac{1}{2}}
		= \left( \sum_{i \in \naturals} i^{-1} \right)^{\frac{1}{2}}
		= \infty \text{.}
	\end{align}
	We see that in this situation, the fine index $r$ decides whether the conditions \eqref{formula embedding equivalent condition sup} and \eqref{formula embedding equivalent condition} are met. 
	\anmerkung{Geht auch für $r=1+\varepsilon$, da ist dann auch $...=\infty$. Das ist aber mehr Rechenaufwand. }
\end{remark}

\review{
\begin{example}
	Let us illustrate the results for Sobolev spaces. Suppose that $d\in\naturals$, $s>\frac{d}{2}$,  and let $X=\mathbb{T}^d:=\reals^d/\integers^d$ be  the $d$-dimensional torus equipped with the Haar measure. 
	Under these conditions, we can consider the Sobolev RKHS $H=H_2^s(\mathbb{T}^d)$, see e.g. \cite[Theorem 8]{RKHS_on_manifolds}. 
	The associated interpolation spaces of the form \eqref{formula interpolation space} then are Besov spaces. That is 
	\begin{align}
		[\gL^2(\mathbb{T}^d), [H]_\sim]_{\theta,r} = B_{2,r}^{\theta s} (\mathbb{T}^d)
	\end{align}
	where $0<\theta<1$ and $1\leq r \leq \infty$, see e.g. \cite[Section 7.3.1]{Triebel_function_spaces_II}. In addition, the norms on these spaces are equivalent. 
	Consequently, we find that $[H]_\sim^{(\theta,r)} \triangleq B_{2,r}^{\theta s} (\mathbb{T}^d)$ by Theorem \ref{thm equivalent norm}. 
	The eigenvalues $\mu_i$ of the integral operator associated to $H$ are given by the eigenvalues of $(\Id+\Delta)^{-s}$, where $\Delta$ is the Laplace-Beltrami operator, see e.g. \cite[Theorem 8]{RKHS_on_manifolds}. This gives an asymptotical decay of $i^{-2s/d}$. \anmerkung{Nach Weyl-Formel, ggf. Literatur?}
	
	As the kernel $k$ on $H$ satisfies $k(x+t,y+t)=k(x,y)$ for all $x,y,t \in \mathbb{T}^d$, the $(e_i)_{i \in I}$ are the $\gL^2(\mathbb{T}^d)$-normalised Fourier basis functions, see e.g. \cite[Lemma 4.2]{Ingo_Ziegel}\anmerkung{Meine Masterarbeit Prop. 2.1.4, nutze hier aber die $\sin$ $\cos$-Basis da reellwertig?}. Consequently, for an $f \in \gL^2(\mathbb{T}^d)$, the associated $b_i$  with $f=\sum_{i \in I} b_i [e_i]_\sim$ are the Fourier coefficients. 
	Hence the power spaces $[H]_\sim^{(\theta,r)}$ contain all $f \in \gL^2(\mathbb{T}^d)$ with Fourier coefficients $b_i$ satisfying 
	\begin{align}
		\left( \sum_{j \in \integers} \left( \sum_{i \in M_j} b_i^2 \mu_i^{-\theta} \right)^\frac{r}{2} \right)^\frac{1}{r} < \infty ,
	\end{align}
	where the latter formula also defines a norm equivalent to the one of $B_{2,r}^{\theta s} (\mathbb{T}^d)$. 
	Such equivalent Besov space norms based on the Fourier coefficients are known in the literature \cite[Section 1]{Besov_spaces_basis}, giving an alternative to the common wavelet approach to Besov spaces\todo{Eine der Quellen vom Reviwer (Gine Nickl oder Triebel)}. 
	
	Note that a similar approach is possible for other domains such as the spheres $\mathbb{S}^d$, which find an important application e.g. in the investigation of neural tangent kernels \cite{Max_NTK}. 
\end{example}
}
\review{
Real interpolation spaces of general Besov spaces do not necessarily stay within the Besov space scale. However, in our setting, we only consider spaces of the form $B_{2,r}^{\theta s}$ for $1\leq r \leq \infty$ and $s>\frac{d}{2}$, where $d$ is the dimension of the base space. 
Interpolation of such spaces again yields spaces of the form $B_{2,r}^{\theta s}$. More generally, we find that interpolation of power spaces gives power spaces again: 
\begin{corollary}\label{corollary power space reiteration}
	Let $0 < \theta < 1$, $1 \leq r \leq \infty$, and suppose that $H$ is as in Assumption \refirgendwo{assumption RKHS}{H}. Then 
	\begin{align}
		\left[[H]_\sim^{(\theta_1,r_1)},[H]_\sim^{(\theta_2,r_2)}\right]_{(\theta,r)} = [H]_\sim^{(\eta,r)}
	\end{align}
	where $\eta=(1-\theta)\theta_1 + \theta \theta_2$. 
	The norms on these spaces are equivalent. 
\end{corollary}
}

\todo[inline,color=cyan]{Bis hier Veränderungen/neue Bemerkungen}

\kommentar{
	Condition \eqref{formula embedding equivalent condition} in Theorem \ref{thm embedding equivalence} translates to 
	\begin{align}
		\sum_{j \in \integers} \left( \sum_{i \in M_j} e_i(x)^2 \mu_i^{\theta} \right)^{\frac{r'}{2}} 
		\leq \kappa^{r'}
	\end{align}
	if $1\leq r' < \infty$ and 
	\begin{align}
		\sup_{j \in \integers} \left( \sum_{i \in M_j} e_i(x)^2 \mu_i^{\theta} \right)^{\frac{1}{2}}
		\leq \kappa
	\end{align}
	if $r'=\infty$. 
}

\kommentar{
	Wenn die Einbettung nach $\gL^\infty$ existiert, $\nu$ ein Wahrscheinlichkeitsmaß ist, und wenn $\frac{r'}{2}\leq1$ (also $r \geq 2$), dann gilt 
\begin{align}
	\infty
	&> \sum_{j \in \integers} \left( \sum_{i \in M_j} e_i(x)^2 \mu_i^{\theta} \right)^{\frac{r'}{2}} \\
	&= \sum_j  \left\Vert \sum_i e_i(\cdot)^2 \mu_i^\theta \right\Vert_{\gL^{\frac{r'}{2}}}^{\frac{r'}{2}}\\
	&\geq \sum_j  \left\Vert \sum_i e_i(\cdot)^2 \mu_i^\theta \right\Vert_{\gL^{1}}^{\frac{r'}{2}}\\
	&= \sum_j \left( \int \sum_i e_i(x)^2 \mu_i^\theta d \nu(x)\right)^{\frac{r'}{2}} \\
	&= \sum_j \left(\mu_i^\theta\right)^{\frac{r'}{2}}
\end{align}
}

	\section{Application to statistical learning} \label{section application RERM}
Let us introduce some basic notions from statistical learning theory, where we refer to \cite{svm} for details. 	We assume that $(X,\mathcal{A})$ is a measurable space, and that $Y\subset \reals$ is bounded. Suppose there is a probability measure $P$ on $X \times Y$, wherefrom only some samples $(x_1,y_1),...,(x_n,y_n)$ are known. 
Typically, one searches for a prediction function $f:X \rightarrow \reals$ that predicts an 
$Y$-value associated to a given $X$-value with low risk 
\begin{align}
	\mathcal{R}_{L,P}(f) := \int_{X \times Y} L(y,f(x)) dP(x,y) \text{.}
\end{align}
Hereby, $L: \reals \times \reals \rightarrow \reals_{\geq 0}$ is the loss function, such as the least squares loss $L(y,t)=(y-t)^2$. 
A function which minimises the risk over all prediction functions $f:X \rightarrow \reals$ is called Bayes decision function and denoted by $f_{L,P}^*$. 

To find a good prediction function, one can choose an RKHS $H$ on $X$, a $\lambda>0$ and perform a Tikhonov-regularised minimisation of the risk, that is  
\begin{align} \label{formula SVM}
	f_{P,\lambda} \in \mathrm{argmin}_{f  \in \review{H}} \; \lambda \Vert f \Vert_H^2 + \mathcal{R}_{L,P} (f)
\end{align}
where $P$ is in practice replaced by the empirical probability measure assigned to the given data. 
As done in \cite{spectral_regularisation}, we assume in this section that the kernel of $H$ is bounded by $1$. 

Assuming that $L$ is the least squares loss, $f_{P,\lambda}$ can alternatively be found by the eigen-decomposition of $H$, see \cite[Theorem 5.4]{general_mercer}: \anmerkung{Dort: $\rho$ ist unser $P$, $f^{\lambda}$ unser $f_{P,\lambda}^*$} \anmerkung{Richige Quelle? Oder \cite{spectral_regularisation}? Oder was anderes?}
\begin{align} \label{formula spectral Tikhonov}
	f_{P,\lambda} = \sum_{i \in I} \frac{\mu_i}{\lambda + \mu_i} \skalarprodukt{f_{L,P}^*}{[e_i]_{\sim}}_{\gL^2(P_X)} e_i
\end{align}
where we assume that $P_X$ is the marginal distribution, that $Y=[-1,1]$ and that $H$, the eigenvalues $\mu_i$ and the basis $e_i$ are as in Assumption \refirgendwo{assumption RKHS}{H}. 
By this reformulation, we can describe various other methods of regularised minimisation than \eqref{formula SVM} by replacing the  term $ \frac{\mu_i}{\lambda + \mu_i}$ with a more general function $\mu_i g_\lambda (\mu_i)$. The following conditions taken from \cite{spectral_regularisation} ensure that a reasonable algorithm emerges. 
\begin{definition} \label{def admissible filter function}
	A family of functions $g_\lambda: [0,1] \rightarrow \reals$ for $0 < \lambda \leq 1$ satisfying 
	\begin{align}
		\sup_{0 < \mu \leq 1} \sup_{0\leq \lambda \leq 1} | \lambda g_\lambda (\mu) | &< \infty \\
		\sup_{0 < \mu \leq 1} \sup_{0 \leq \lambda \leq 1} |\mu g_\lambda (\mu)| &< \infty  \label{formula spectral filter bounded} \\
		\lim_{\lambda \rightarrow 0} \mu g_\lambda (\mu) &= 1
	\end{align}
	for all $0 \leq \mu \leq 1$ is called an admissible filter function, if there is a constant $\alpha_0 > 0$ such that for all $0 < \alpha \leq \alpha_0$
	\begin{align}
		\sup_{0 < \mu \leq 1} \sup_{0 \leq \lambda \leq 1} | 1- \mu g_\lambda (\mu) | \mu^\alpha \lambda^{-\alpha} < \infty \text{.}
	\end{align}
\end{definition}

Let us now replace the term $\frac{\mu_i}{\lambda + \mu_i}$ in \eqref{formula spectral Tikhonov} by the more general $\mu_i g_\lambda (\mu_i)$, that is 
\begin{align} \label{formula spectral regularisation}
	f_{P,\lambda,g} := \sum_{i \in I} \mu_i g_\lambda (\mu_i) \skalarprodukt{f_{L,P}^*}{[e_i]_{\sim}}_{\gL^2(P_X)} e_i \text{.}
\end{align}
This converges in $\gL^2(P_X)$ if $f_{L,P}^* \in \gL^2(P_X)$, as $\mu_i g_\lambda (\mu_i)$  is bounded by \eqref{formula spectral filter bounded}. 
For example, to describe the Landweber iteration, that is a form of gradient descent with early stopping, we can choose 
\begin{align}
	g_t (\mu) := \sum_{i=0}^{t-1}(1-\mu)^i
\end{align}
by using $\lambda:=t^{-1}$ with $t \in \naturals$. 
For details on the Landweber iteration and for further examples of various algorithms arising from different choices of $g_\lambda$, see \cite{spectral_regularisation}. 


We can estimate the \enquote{approximation error}, in our case this is $f_{P,\lambda,g} - f_{L,P}^*$  measured in interpolation space norms. This can be viewed as the distance of the optimal Bayes decision function $f_{L,P}^*$ to the result of the regularised empirical risk minimisation $f_{P,\lambda,g}$ in the case of full knowledge of the distribution $P$. 
\begin{proposition} \label{prop regularisation error in interp norm}
	Let $X$ be a measurable space, $P$ a measure on $X \times [-1,1]$ and $k$ a measurable kernel on $X$ with RKHS $H$ that is compactly embedded into $\gL^2(P_X)$, where the eigenvalues of $T_k$ are bounded by $1$. \anmerkung{$\mu \leq 1$ wegen Def \ref{def admissible filter function}}
	Let $0< \beta < 1$, $0< \theta < 1$\anmerkung{Anders als bei Ingo: $\theta>\beta$ erlaubt} and $1 \leq r \leq \infty$. If $f_{L,P}^* \in [\gL^2(P_X),[H]_{\sim}]_{\beta,r}$, the following holds: 
	\begin{enumerate}
		\item There is a constant $C_1$	independent of $\lambda$ and $g_{\lambda}$ such that 
		\begin{align}
			\Vert [f_{P, \lambda,g}]_{\sim} - f_{L,P}^* \Vert_{[\gL^2(X),[H]_{\sim}]_{\theta,r}}
			\leq C_1 \sup_{i \in I} \left( \big| 1 - \mu_i g_\lambda (\mu_i) \big| \mu_i^{\frac{\beta-\theta}{2}} \right) \Vert f_{L,P}^* \Vert_{[\gL^2(X),[H]_{\sim}]_{\beta,r}} \text{.}
		\end{align}
		\item There is a constant $C_2$ independent of $\lambda$ and $g_{\lambda}$ such that 
		\begin{align} \label{formula estimate f_P,lambda,g}
		\Vert [f_{P, \lambda,g}]_{\sim} \Vert_{[\gL^2(X),[H]_{\sim}]_{\theta,r}}
		\leq C_2 \sup_{i \in I} \left( \big| \mu_i g_\lambda (\mu_i) \big| \mu_i^{\frac{\beta-\theta}{2}} \right) \Vert f_{L,P}^* \Vert_{[\gL^2(X),[H]_{\sim}]_{\beta,r}} \text{.}
		\end{align}
	\end{enumerate}
\end{proposition}

Suppose that in addition to the requirements of Proposition \ref{prop regularisation error in interp norm}, we know that the embedding $[\gL^2(\nu), [H]_\sim]_{\theta,r} \hookrightarrow \gL^\infty(\nu)$ is continuous with operator norm $\kappa$, that is 
\begin{align}
	\Vert [f_{P, \lambda, g}]_{\sim} \Vert_{\gL^\infty(\nu)} 
	\leq \kappa \Vert [f_{P, \lambda, g}]_{\sim} \Vert_{[\gL^2(X),[H]_{\sim}]_{\theta,r}}
\end{align}
which yields in combination with \eqref{formula estimate f_P,lambda,g} that
\begin{align}
	\Vert [f_{P, \lambda,g}]_{\sim} \Vert_{\gL^\infty(\nu)}
	\leq C_3 \sup_{i \in I} \left( \big| \mu_i g_\lambda (\mu_i) \big| \mu_i^{\frac{\beta-\theta}{2}} \right) \Vert f_{L,P}^* \Vert_{[\gL^2(X),[H]_{\sim}]_{\beta,r}} \text{.}
\end{align}
Such estimates are play an important role in the statistical analysis of kernel-based learning algorithms, see e.g. \cite[Theorem 7.23]{svm}. 
Note that the required existence of the embedding to $\gL^\infty(\nu)$ may be checked via Theorem \ref{thm embedding equivalence}.

\begin{remark}
	We elaborate Proposition \ref{prop regularisation error in interp norm} for two concrete regularisation settings. To this end, we assume that $\beta > \theta$, and hence $\gamma := \frac{\beta-\theta}{2}>0$. 
	\begin{enumerate}
		\item \textbf{Tikhonov regularisation.} Here we have $g_\lambda(\mu)=\frac{1}{\lambda + \mu}$. For the supremum in Proposition \ref{prop regularisation error in interp norm}, we find that  
		\begin{align}
			\sup_{i \in I} \left( \big| 1 - \mu_i g_\lambda (\mu_i) \big| \mu_i^{\gamma} \right)
			= \sup_{i \in I} \left( \frac{\lambda \mu_i^\gamma}{\lambda + \mu_i} \right)
			&\leq \sup_{0 < \mu \leq 1} \left( \frac{\lambda \mu^\gamma}{\lambda + \mu} \right) \\
			&\leq \lambda^\gamma \frac{\left( \frac{\gamma}{1-\gamma}\right)^{\gamma}}{1+\frac{\gamma}{1-\gamma}} \\
			&= \lambda^\gamma \gamma^\gamma (1-\gamma)^{1-\gamma} \text{,}
		\end{align}
		which can be seen by finding the zero of the derivative of the expression in question. 
		Hence for the estimate in Proposition \ref{prop regularisation error in interp norm}, we get 
		\begin{align}
			\Vert [f_{P, \lambda,g}]_{\sim} - f_{L,P}^* \Vert_{[\gL^2(X),[H]_{\sim}]_{\theta,r}}
			\leq C_3 \lambda^\gamma \Vert f_{L,P}^* \Vert_{[\gL^2(X),[H]_{\sim}]_{\beta,r}} \text{,}
		\end{align}
		where we additionally use that $\gamma^\gamma (1-\gamma)^{1-\gamma} \leq 1$. 
		\item \textbf{Landweber iteration.} We have $g_t (\mu) = \sum_{i=0}^{t-1}(1-\mu)^i$ with $\lambda=t^{-1}$. 
		For $\mu < 1$ we reformulate this by the geometric sum as $g_t = \frac{1-(1-\mu)^t}{\mu}$ to find that 
		\begin{align}
			\sup_{i \in I} \left( \big| 1 - \mu_i g_\lambda (\mu_i) \big| \mu_i^{\gamma} \right)
			= \sup_{i \in I} (1- \mu_i)^t \mu_i^{\gamma}
			\leq \sup_{0 < \mu < 1} (1- \mu)^t \mu^{\gamma}
			= \left( \frac{t}{\gamma+t} \right)^t \left( \frac{\gamma}{\gamma + t} \right)^\gamma \text{,}
		\end{align}
		which can be seen by finding the zero of the derivative of the expression in question. 
		Hence for the estimate in Proposition \ref{prop regularisation error in interp norm}, we get 
		\begin{align}
			\Vert [f_{P, \lambda,g}]_{\sim} - f_{L,P}^* \Vert_{[\gL^2(X),[H]_{\sim}]_{\theta,r}}
			\leq C_3 \left( \frac{\gamma}{\gamma + t} \right)^\gamma \Vert f_{L,P}^* \Vert_{[\gL^2(X),[H]_{\sim}]_{\beta,r}} \text{,}
		\end{align}
		where we note that estimating $\left( \frac{t}{\gamma+t}\right)^t \leq 1$ does not affect asymptotic properties in $t$. 
	\end{enumerate}
\end{remark}

	\section{Proofs} \label{section proofs}

We can describe interpolation spaces of weighted $\gL^p$-spaces by an equivalent norm via the following theorem. For the sake of completeness, we cite this result using our notation here. 

\begin{theorem} \label{thm Gilbert}
	Let $(X, \mathcal{A},\nu)$ be a measure space, and $w: X \rightarrow [\varepsilon, \infty)$  be measurable\anmerkung{Das nur um die vereinfachte Definition einzuhalten, also $\gL^p(w) \hookrightarrow \gL^p(\nu)$}, where $\varepsilon>0$ is some constant. 
	Furthermore, let $s>1$, $0 < \theta < 1$, and $1 \leq r < \infty$. Then 
	\begin{align} \label{formula Gilbert norm}
		\Vert f \Vert = \left(\sum_{j \in \integers} \left( s^{-j \theta} \Vert f \cdot \mathbbm{1}_{\tilde{M}_j} \Vert_{\gL^p(\nu)} \right)^r \right)^{\frac{1}{r}} \text{,}
	\end{align}
	where $\tilde{M}_j := \{x \in X \, | \,  s^{-j} < w(x)^{\frac{1}{p}} \leq s^{-j+1} \}$, is an equivalent norm on $[\gL^p(\nu), \allowbreak \gL^p(w)]_{(\theta,r)}$. 
	An analogous result holds for $r=\infty$. 
\end{theorem}
\begin{proof}
	The assertion directly follows from \cite[Theorem 3.3]{gilbert}\anmerkung{\cite[Theorem 3.7]{gilbert} machts eigentlich direkter, aber da ist n Fehler. Der liegt an der Def. von $X_t$ dort vor 3.3, da wird $\frac{1}{t}$ verwendet. Das wird später nicht beachtet.  }, where we set $\mathscr{X}$ as $\gL^p(\nu)$ and use the weight $\omega=w^{\frac{1}{p}}$, as the author applies the weight to the functions instead of the measures as we do here. Finally, we need to recall that the $K$-method and $J$-method used in \cite{gilbert} yield equivalent norms on the interpolation space. 
\end{proof}

Let us adapt this result to our needs. To this end, we generalise the definition of $M_j$ from \eqref{formula def Mj} to a general measure space $(X,\mathcal{A},\nu)$ and a measurable \enquote{weight} function $w: X \rightarrow (0,\infty)$: 
\begin{align}
	M_j:= \left\{x \in X \, \middle| \, \frac{1}{2^{j+1}} \leq \frac{1}{w(x)} < \frac{1}{2^{j}} \right\}
\end{align}
\begin{lemma} \label{lemma weighted Lp interpolation}\anmerkung{Wird nur für $p=2$ angewendet, kann man also auch darauf einschränken}
	Let $(X, \mathcal{A},\nu)$ be a measure space, let $\varepsilon>0$, and $w: X \rightarrow [\varepsilon, \infty)$ be measurable\anmerkung{Das nur um die vereinfachte Definition einzuhalten, also $\gL^p(w) \hookrightarrow \gL^p(\nu)$}. 
	Furthermore, let $0 < \theta < 1$, $1 \leq r < \infty$, and let 
	\begin{align}
		c_j(f):= \left( \int_{M_j} |f|^p w^{\theta} \gd \nu \right)^{\frac{1}{p}}
	\end{align}
	for $f \in [\gL^p(\nu), \gL^p(w)]_{\theta,r}$ and $j \in \integers$. 
	Then an equivalent norm on $[\gL^p(\nu), \gL^p(w)]_{\theta,r}$ is given by 
	\begin{align}
		\Vert f \Vert := \Vert (c_j(f))_{j \in \integers} \Vert_{\ell^r} \text{.}
	\end{align}
\end{lemma}
\begin{proof}
	Let $s$ and $\tilde{M}_j$ be as in Theorem \ref{thm Gilbert} 
	and denote the norm in \eqref{formula Gilbert norm} by $\dreistrichnorm \cdot \dreistrichnorm$. 
	For $s:=2^{\frac{1}{p}}$ we get 
	\begin{align}
		\tilde{M}_{-j}
		= \{x \in X \, | \, s^{j} < w(x)^{\frac{1}{p}} \leq s^{j+1} \}
		= \{x \in X \, | \, 2^{-(j+1)} \leq w(x)^{-1} < 2^{-j} \}
		= M_j \text{.}
	\end{align}
	Hence performing an index swap from $j$ to $-j$ yields 
	\begin{align} 
		\dreistrichnorm f \dreistrichnorm
		=\left( \sum_{j \in \integers}  \left( \int_{M_j} 2^{j \theta} |f|^p \gd \nu \right)^{\frac{r}{p}} \right)^{\frac{1}{r}}
	\end{align}
	if $1 \leq r < \infty$, and similarly for $r=\infty$ using the supremum. 
	
	Finally, for $x \in M_j$, we have $2^{j \theta} < w(x)^{\theta} \leq 2^{(j+1)\theta}$. 
	Hence replacing $2^{j \theta}$ by $w(x)^{\theta}$ for $x \in M_j$ changes the expression at most by a factor of $2$. Consequently, $\Vert \cdot \Vert$ and $\dreistrichnorm \cdot \dreistrichnorm$ are equivalent norms. 
\end{proof}

\subsection{Auxiliary results on RKHS}


For an RKHS $H$ as in Assumption \refirgendwo{assumption RKHS}{H}, we considered the space of equivalence classes $[H]_{\sim}$ equipped with the quotient norm. It can be understood by the eigenfunction decomposition via the following lemma. 
\begin{lemma} \label{lemma H sequence space unitary operator}
	Let $H$, $\mu_i$ and $e_i$ be as in Assumption \refirgendwo{assumption RKHS}{H}. Then the mapping defined by 
	\begin{align}
		\phi: [H]_{\sim}  \rightarrow \ell^2(\mu^{-1}), 
		\quad [f]_{\sim} \mapsto \left( \skalarprodukt{[f]_{\sim}}{[e_i]_{\sim}}_{\gL^2(\nu)} \right)_{i \in I}
	\end{align}
	is an isometric isomorphism. 
\end{lemma}
\begin{proof}
	Recall that $(\sqrt{\mu_i} e_i)_i$ is an ONS in $H$. 
	We write $H_I:= \overline{\spann \{\sqrt{\mu_i} e_i | i \in I\}}$ for the subspace spanned by this ONS. 
	By \cite[Formula (17)]{general_mercer}, we know that $H_I=\overline{\bild I_k^*}$, and hence \cite[Formula (19)]{general_mercer} shows $H_I^\perp = \kernn I_k$. 
	If we choose an ONB $(\tilde{e}_j)_j$ of $H_I^\perp$ on a new index set $J$, we thus find 
	\begin{align}
		[H]_\sim &= \left\{ \left[\sum_{i \in I} a_i \sqrt{\mu_i} e_i + \sum_{j \in J} b_j \tilde{e}_j\right]_\sim \, \middle| \, a \in \ell^2(I), b \in \ell^2(J) \right\} \\
		&= \left\{ \left[\sum_{i \in I} a_i \sqrt{\mu_i} e_i \right] \, \middle| \, a \in \ell^2(I) \right\} \\
		&= [H_I]_\sim \text{.}
	\end{align} 
	Since $\tilde{I}_k:=I_k|_{H_I \rightarrow [H]_\sim} $ is bijective by the considerations above, we see that for all $f \in [H_I]_\sim$ there is exactly one $h_f \in H_I$ such that $[h_f]_\sim = \tilde{I}_k h_f = f$. This gives 
	\begin{align}
		\Vert f \Vert_{[H_I]_\sim}
		= \inf \{ \Vert h \Vert_{H} \, : \, h \in H_I \text{ with } [h]_\sim=f\}
		= \Vert h_f \Vert_H \text{.}
	\end{align}
	Moreover, for $f \in [H]_\sim$ we find 
	\begin{align}
		\Vert f \Vert_{[H]_\sim}^2
		&= \inf \{ \Vert h \Vert_H^2 \, : \, h \in H \text{ with } [h]_\sim=f\} \\
		&= \inf \{ \Vert h_I + h_0 \Vert_H^2 \, : \, h_I \in H_I \text{ and } h_0 \in \kernn I_k \text{ with } [h_I+ h_0]_\sim = f \} \\
		&= \inf \{ \Vert h_I \Vert^2_H + \Vert h_0 \Vert_H^2 \, : \, h_I \in H_I \text{ and } h_0 \in \kernn I_k \text{ with } [h_I+ h_0]_\sim = f \} \\
		&= \inf \{ \Vert h_I \Vert^2_H  \, : \, h_I \in H_I \text{ with } I_k h_I = f \} \\
		&= \Vert h_f \Vert^2_H \text{.}
	\end{align}
	Combining both considerations, we find $\Vert \, \cdot \, \Vert_{[H]_\sim}=\Vert \, \cdot \, \Vert_{[H_I]_\sim}$. 
	
	Now for $h \in H_I$, let $a_i:=\skalarprodukt{h}{\sqrt{\mu}_i e_i}_H$ and $b_i := \skalarprodukt{[h]_\sim}{([e_i]_\sim)_{i \in I}}_{\gL^2(\nu)}$. 
	Then we have
	\begin{align}
		\sum_{i \in I} b_i [e_i]_\sim
		= [h]_\sim
		= \left[\sum_{i \in I} a_i \sqrt{\mu_i} e_i \right]_\sim
		= \sum_{i \in I} a_i \sqrt{\mu_i} [e_i]_\sim \text{,}
	\end{align}
	and thus $a_i \sqrt{\mu_i} = b_i$. 
	Since $a \in \ell^2(I)$, we get 
	\begin{align}
		\Vert b \Vert_{\ell^2(\mu^{-1})}^2
		= \sum_{i \in I} b_i^2 \mu_i^{-1}
		= \sum_{i \in I} \left(\frac{b_i}{\sqrt{\mu_i}}\right)^2
		= \sum_{i \in I} a_i^2
		= \Vert h \Vert_{H_I}^2 \text{.}
	\end{align}
	Consequently, the map $h \mapsto \skalarprodukt{[h]_\sim}{[e_i]_\sim}_{\gL^2(\nu)}$ is isometric. By the arguments above, it is not hard to see that it is also surjective, thus it is an isometric isomorphism. 
\end{proof}

The following lemma translates the interpolation space norm of $[\gL^2(\nu), [H]_{\sim}]_{(\theta,r)}$ to the eigenfunction decomposition of the RKHS. 
\begin{lemma} \label{lemma sequence interpolation space equivalence}
	For $0<\theta<1$, $1\leq r \leq \infty$ and $\nu$, $H$, $e_i$ and $\mu_i$ as in Assumption \refirgendwo{assumption RKHS}{H}, the following map is an isometric isomorphism: 
	\begin{align}
		[\gL^2(\nu), [H]_{\sim}]_{(\theta,r)} \rightarrow [\ell^2(I), \ell^2(\mu^{-1})]_{(\theta,r)}, 
		\quad f \mapsto (\skalarprodukt{f}{[e_i]_{\sim}})_{i \in I}
	\end{align}
\end{lemma}
\begin{proof}
	Let us define $\mathcal{H}:= \overline{[H]_{\sim}}^{\,\gL^2(\nu)}$ and equip this space with the $\gL^2(\nu)$-inner product. 
	An application of Lemma \ref{lemma interpolation space direct sum} to this orthogonal decomposition yields 
	\begin{align}
		\left[\gL^2(\nu), [H]_{\sim}\right]_{\theta,r}
		= \left[\mathcal{H} \oplus \mathcal{H}^\perp, [H]_{\sim}\right]_{\theta,r}
		= \left[\mathcal{H}, [H]_{\sim}\right]_{\theta,r}
	\end{align}
	as $[H]_{\sim} \hookrightarrow \mathcal{H}$. \anmerkung{Nach Annahme kompakt eingebettet}
	Now let us consider the operator 
	\begin{align}
		\phi: \mathcal{H} \rightarrow \ell^2(I), \quad x \mapsto \left(\skalarprodukt{x}{[e_i]_{\sim}}_{\gL^2(\nu)}\right)_{i \in I} \text{.}
	\end{align}
	Note that this operator is an isometric isomorphism. 
	The restricted operator  
	\begin{align}
		\phi |_{[H]_{\sim} \rightarrow \ell^2(\mu^{-1})}
	\end{align}
	is again an isometric isomorphism by Lemma \ref{lemma H sequence space unitary operator}. 
	Hence Lemma \ref{lemma interpolation space isometric isomorphism} shows that the operator restricted to the interpolation spaces
	\begin{align}
		\phi |_{\left[\mathcal{H}, [H]_{\sim}\right]_{\theta,r} \rightarrow \left[ \ell^2(I),  \ell^2(\mu^{-1}) \right]}
	\end{align}
	is again an isometric isomorphism. 
	This proves the assertion. 
\end{proof}

\subsection{Proof of Theorem \ref{thm equivalent norm}}


\begin{proof}[Proof of Theorem \ref{thm equivalent norm}] \labelirgendwo{proof thm equivalent norm}
	We apply Lemma \ref{lemma weighted Lp interpolation} to the counting measure on $I$ and the weight sequence $\mu^{-1}$ to find that the identity mapping $J: [\ell^2(I), \ell^2(\mu^{-1})]_{(\theta,r)} \rightarrow \ell^{(\mu,\theta,r)}$ is a norm equivalence. 
	
	Consider the isometric isomorphism  $\alpha: [\gL^2(\nu),[H]_\sim]_{\theta,r} \rightarrow [\ell^2(I), \ell^2(\mu^{-1})]_{(\theta,r)}$ from Lemma \ref{lemma sequence interpolation space equivalence}. It maps functions to their $\gL^2(\nu)$-Fourier coefficients with respect to the ONS $e_i$. 
	
	The map $\beta: \ell^{(\mu,\theta,r)} \rightarrow [H]_\sim^{(\theta,r)}$ given by $b \mapsto \sum_{i \in I} b_i [e_i]_\sim$ is an isometric isomorphism by Definition \ref{def power space}.
	
	We note that for $f \in [\gL^2(\nu),[H]_\sim]_{\theta,r}$, we have \anmerkung{Letzte Gleichheit gilt da $\beta$ isom. Iso. ist und damit das ONS ausreicht}
	\begin{align}
		\beta(J(\alpha(f)))
		= \beta(\alpha(f)) 
		= \beta((\skalarprodukt{f}{[e_i]_\sim})_{i \in I})
		= \sum_{i \in I} \skalarprodukt{f}{[e_i]_\sim} [e_i]_\sim
		= f \text{,}
	\end{align}
	which shows that the following diagram commutes: 
	\begin{center}
		\begin{tikzcd}
		{[\gL^2(\nu),[H]_\sim]_{\theta,r}} \arrow[rr, "\triangleq"',"\Id"] \arrow[dd, "\simeq","\alpha"'] &  & {[H]_\sim^{(\theta,r)}}                  \\
		&  &                    \\
		{[\ell^2(I), \ell^2(\mu^{-1})]_{(\theta,r)}} \arrow[rr, "\triangleq","J"']                 &  & {\ell^{(\mu,\theta,r)}} \arrow[uu, "\simeq","\beta"']
	\end{tikzcd}
	\end{center}
	This proves the assertion. 
\end{proof}

\begin{remark} \label{remark convergence in L2}
	By Lemma \ref{lemma interpolation space subsets}, we know that $[\gL^2(\nu), [H]_{\sim}]_{(\theta,r)} \hookrightarrow \gL^2(\nu)$. Hence we can apply Theorem \ref{thm equivalent norm} to see that for the power spaces, we have $[H]_\sim^{(\theta,r)} \hookrightarrow \gL^2(\nu)$ as well. However, the operator norms of these two embedding operators might differ by a constant, as the norms are equivalent, but not necessarily equal. 
\end{remark}

\begin{lemma} \label{lemma convergence in interp norm}
	Let $0<\theta<1$ and $1 \leq r \leq \infty$. If $b \in \ell^{(\mu,\theta,r)}$, then $\sum_{i \in I} b_i [e_i]_{\sim}$ converges in the power space $[H]_\sim^{(\theta,r)}$. 
\end{lemma}
\begin{proof}
	For finite $I$, the assertion is clear. Hence let us assume $I=\naturals$. 
	Let $f := \sum_{i \in I} b_i [e_i]_{\sim}$, where convergence is in $\gL^2(\nu)$ at first. 
	To show convergence in $[H]_\sim^{(\theta,r)}$-norm, we notice that $\ell^{(\mu,\theta,r)} \triangleq [\ell^2(I), \allowbreak \ell^2(\mu^{-1})]_{(\theta,r)}$ by Lemma \ref{lemma weighted Lp interpolation}.  Applying Lemma \ref{lemma interpolation space subsets} yields $\ell^{(\mu,\theta,r)} \hookrightarrow \ell^2(I)$. If $C$ denotes the corresponding operator norm, we thus find 
	\begin{align}
		\left\Vert f - \sum_{i=1}^n b_i [e_i]_{\sim} \right\Vert_{[H]_\sim^{(\theta,r)}}
		= \Vert (b_i \mathbbm{1}_{i \geq n+1}) \Vert_{(\mu,\theta,r)}
		\leq C \Vert (b_i \mathbbm{1}_{i \geq n+1}) \Vert_{\ell^2(I)}
		 \xrightarrow{n \rightarrow \infty} 0 \text{.}
	\end{align}
\end{proof}

\subsection{Liftings}

Let us introduce a measure-theoretic technique that will play a role in the subsequent proofs. 
To this end, we consider the map 
\begin{align}
	\mathcal{L}^\infty (X) \rightarrow \gL^{\infty} (\nu), \quad f \mapsto [f]_{\sim}
\end{align}
which assigns to each $f$ its equivalence class of functions. 
This mapping is surjective by definition of $\gL^{\infty} (\nu)$, hence it possesses a, not necessarily unique, right inverse. \anmerkung{Das braucht das Auswahlaxiom}
However, this reasoning does not yield further properties of such a right inverse. Nonetheless, by means of lifting theory, one can show that a bounded linear right inverse exists. 

\begin{theorem} \label{thm lifting}
	Let $\nu$ be a $\sigma$-finite measure on the measurable space $(X,\mathcal{A})$ such that $\mathcal{A}$ is $\nu$-complete. 
	Then there exists a bounded, linear, and injective operator 
	\begin{align}
		\phi: \gL^{\infty} (\nu) \rightarrow \mathcal{L}^\infty (X)
	\end{align}
	that satisfies $\Vert \phi \Vert \leq 1$ and 
	\begin{align}
		[\phi ([f]_{\sim})]_{\sim} = [f]_{\sim}
	\end{align}
	for all $f \in \mathcal{L}^{\infty} (\nu)$. 
\end{theorem}
The proof is based on lifting theory as developed in \cite{ionescu_tulcea}. For a derivation of the specific result, see \cite[Equation (58)]{general_mercer}. 

In other words, for each equivalence class in $\gL^{\infty} (\nu)$, the operator $\phi$ chooses a bounded representative  in a linear and continuous way. 


\subsection{Proofs of Theorems \ref{thm interp space BSF}, \ref{thm interp space BSF with topology}, and \ref{thm embedding equivalence}}


Let us now introduce a setup which makes it possible to apply the results for direct sums of Banach spaces from Appendix \ref{appendix direct sums}. 

\begin{remark} \label{remark isometric isomorphisms}
	Let $0<\theta<1$ and $1\leq r \leq \infty$ and suppose that $\mu=(\mu_i)_{i \in \naturals}$ is a non-increasing sequence of positive real numbers.  
	Let us define the spaces $E_j := \ell^2( (\mu_i^{-\theta})_{i \in M_j})$. 
	
	Then for $j \in \integers$, let us consider the projection maps $\pi_j$, that map any sequence $(b_i)_{i \in I}$ to the subsequence $(b_i)_{i \in M_j}$. 
	More specifically, we define 
	\begin{align}
		\pi_j^{(r)}: \ell^{(\mu,\theta,r)} \rightarrow E_j, 
		\quad b \mapsto (b_i)_{i \in M_j} \text{.}
	\end{align}
	\labelirgendwo{proof ell spaces Banach}Then by Definition \ref{def ell spaces} and Appendix \ref{appendix direct sums}, we know that 
	\begin{align} \label{formula isometric isomorphism ell spaces to direct sum}
		\pi^{(r)}: \ell^{(\mu,\theta,r)} \rightarrow {\bigoplus_{j \in \integers}}^{(r)} E_j, 
		\quad b \mapsto (\pi_j^{(r)}(b))_{j \in \integers}
	\end{align}
	is an isometric isomorphism, and hence the spaces $\ell^{(\mu,\theta,r)}$ are Banach spaces. 
	
	Now note that by the definition of the $E_j$-inner product, we have 
	\begin{align}
		\skalarprodukt{\pi_j(b)}{\pi_j(a)}_{E_j} = \sum_{i \in M_j} a_i b_i \mu_i^{-\theta}
	\end{align}
	for sequences $a,b$ of real numbers. 
	Hence we find that 
	\begin{align}
		\sum_{j \in \integers} \skalarprodukt{\pi_j(a)}{\pi_j(b)}_{E_j}
		= \sum_{j \in \integers} \sum_{i \in M_j} a_i b_i \mu_i^{-\theta}
		= \sum_{i \in I} a_i b_i \mu_i^{-\theta} \label{formula ell spaces dual pairing}
	\end{align}
	whenever one of the sums converges absolutely. 
	By Lemma \ref{lemma dual of direct sum} and the isometric isomorphism \eqref{formula isometric isomorphism ell spaces to direct sum}, this can be viewed as a dual pairing of $\ell^{(\mu,\theta,r)}$ and $\ell^{(\mu,\theta,r')}$ if $\frac{1}{r}+\frac{1}{r'}=1$. 
\end{remark}

This setup now enables us to prove what can be viewed as a kind of Hölder inequality for the $\ell^{(\mu,\theta,r)}$-spaces. 

\begin{lemma} \label{lemma Hölder for ell theta r}
	Let $0<\theta<1$ and suppose that $1\leq r \leq \infty$ with $\frac{1}{r} + \frac{1}{r'}=1$. 
	Let $\mu=(\mu_i)_{i \in I}$ be a non-increasing sequence of positive real numbers.  
	Then for $b \in \ell^{(\mu,\theta,r)}$ and $a \in \ell^{(\mu,\theta,r')}$ we have 
	\begin{align}
		\sum_{i \in I} | a_i b_i | \mu_i^{-\theta} 
		\leq \Vert a \Vert_{\ell^{(\mu,\theta,r')}} \Vert b \Vert_{\ell^{(\mu,\theta,r)}} \text{.}
	\end{align}
\end{lemma}
\begin{proof} \labelirgendwo{proof lemma Hölder for ell theta r}
	By Remark \ref{remark isometric isomorphisms}, for all $a \in \ell^{(\mu,\theta,r')}$ and all $b \in \ell^{(\mu,\theta,r)}$, we have 
	\begin{align}
		\sum_{i \in I} |a_i b_i| \mu_i^{-\theta} 
		&= \sum_{j \in \integers} \sum_{i \in M_j} |a_i b_i| \mu_i^{-\theta} \\
		&\leq \sum_{j \in \integers} \left( \sum_{i \in M_j} a_i^2 \mu_i^{-\theta} \right)^{\frac{1}{2}} \left( \sum_{i \in M_j} b_i^2 \mu_i^{-\theta} \right)^{\frac{1}{2}} \\
		&\leq \left( \sum_{j \in \integers} \left( \sum_{i \in M_j} a_i^2 \mu_i^{-\theta} \right)^{\frac{r'}{2}} \right)^{\frac{1}{r'}} \left( \sum_{j \in \integers} \left( \sum_{i \in M_j} b_i^2 \mu_i^{-\theta} \right)^{\frac{r}{2}} \right)^{\frac{1}{r}} \\
		&= \Vert a \Vert_{\ell^{(\mu,\theta,r')}} \Vert b \Vert_{\ell^{(\mu,\theta,r)}} \text{,}
	\end{align}
	where we first applied Cauchy-Schwarz and then Hölder's inequality. 
\end{proof}

The following lemma is essentially given by an application of Lemma \ref{lemma finite norm by dual evaluation} to the setting introduced in Remark \ref{remark isometric isomorphisms}. 
\begin{lemma} \label{lemma finiteness if all dual pairings finite} \anmerkung{Verallg. von \cite[Thm. 6.1]{general_mercer}}
	Let $0<\theta<1$ and suppose that $1 \leq r \leq \infty$ with $\frac{1}{r}+\frac{1}{r'}=1$. Let furthermore $\mu=(\mu_i)_{i \in I}$ be a non-increasing sequence of positive real numbers and $a=(a_i)_{i \in I}$ be a sequence of real numbers. 
	Then the following statements are equivalent for all $C \geq 0$: 
	\begin{enumerate}
		\item For all sequences $b \in \ell^{(\mu,\theta,r)}$, the sum $\sum_{i \in I} a_i b_i \mu_i^{-\theta} $ converges and it holds that 
		\begin{align} \label{formula dual pairing}
			\left| \sum_{i \in I} a_i b_i \mu_i^{-\theta} \right| \leq C \Vert b \Vert_{\mu, \theta, r} \text{.}
		\end{align}
		
		\item It holds that $a \in \ell^{(\mu,\theta,r')}$ with norm $\Vert a \Vert_{(\mu,\theta,r')} \leq C$. 
	\end{enumerate} 
\end{lemma}
\begin{proof} \labelirgendwo{proof lemma finiteness if all dual pairings finite}
	Let us use the notation from Remark \ref{remark isometric isomorphisms}. By \eqref{formula ell spaces dual pairing}, (i) means that for all $b \in \ell^{(\mu, \theta, r)}$, we have 
	\begin{align}
		\left| \sum_{j \in \integers} \skalarprodukt{\pi_j(a)}{\pi_j^{(r)}(b)}_{E_j} \right|
		= \left| \sum_{i \in I} a_i b_i \mu_i^{-\theta} \right|
		\leq C \Vert b \Vert_{(\mu,\theta,r)} 
		= C \Vert \pi^{(r)} (b) \Vert_{\bigoplus_{j \in \integers}^{(r)} E_j}
	\end{align}
	with absolute convergence of the sums by Remark \ref{remark absoulte values in sums}. 
	As $\pi^{(r)}: \ell^{(\mu,\theta,r)} \rightarrow {\bigoplus_{j \in \integers}}^{(r)} E_j$ is an isometric isomorphism, this means that for all $x \in {\bigoplus_{j \in \integers}}^{(r)} E_j$, it holds that 
	\begin{align}
		\left| \sum_{j \in \integers} \skalarprodukt{\pi_j(a)}{x_j}_{E_j} \right| \leq C \Vert x \Vert_{\bigoplus_{j \in \integers}^{(r)} E_j}
	\end{align} 
	where we replaced $\pi^{(r)}(b)$ by $x$ and applied the isometry property on the right-hand side. 
	An application of Proposition \ref{proposition finite norm by dual evaluation with constants} shows that this is the case if and only if $(\pi_j(a))_{j \in \integers} \in \bigoplus_{j \in \integers}^{(r')} E_j$ with norm $\Vert (\pi_j(a))_{j \in \integers} \Vert_{r'} \leq C$. An application of the isometric isomorphism $\pi^{(r')}$ finally shows that this is equivalent to $a \in \ell^{(\mu,\theta,r')}$ with norm $\Vert a \Vert_{(\mu,\theta,r')} \leq C$, which is (ii). 
\end{proof}

\todo[inline,color=cyan]{Ab hier Veränderungen: }

In order to prove Theorem \ref{thm interp space BSF}, we will use liftings from Theorem \ref{thm lifting}. As this requires a complete measure space, which is not necessarily given in Theorem \ref{thm interp space BSF}, we consider the following lemma. 

\begin{lemma} \label{lemma completion}
	Let $(X, \mathcal{A}, \nu)$ be a measure space and $(X, \hat{\mathcal{A}}, \hat{\nu})$ be is its completion. 
	Then the following statements hold true: 
	\begin{enumerate}
		\item For every $\hat{\mathcal{A}}$-measurable $f: X \rightarrow \reals$, there exists a $\mathcal{\nu}$-null set $N$ such that $\mathbbm{1}_{X \setminus N} f$ is $\mathcal{A}$-measurable. 
		\item The operator 
		\begin{align}
			\gL^{\infty} (\nu) \rightarrow \gL^\infty (\hat{\nu}), \quad [f]_\sim \mapsto [f]_{\hat{\sim}}
		\end{align}
		is a well-defined isometric isomorphism, where $[\, \cdot \,]_{\hat{\sim}}$ denotes equivalence classes with respect to $\hat{\nu}$. 
	\end{enumerate}
\end{lemma}
\begin{proof}
	\ada{i} By \cite[Proposition 2.2.5]{Cohn_measure_theory}, there are $\mathcal{A}$-measurable functions $f_0$ and $f_1$ on $X$ with $f_0 \leq f \leq f_1$ and $f_0=f_1$ $\nu$-almost everywhere. Choosing $N=\{x \in X \, | \, f_0(x) \neq f_1(x)\}$ shows that $\mathbbm{1}_{X \setminus N} f = \mathbbm{1}_{X \setminus N} f_0$  is $\mathcal{A}$-measurable. 
	
	\ada{ii} We first prove that the mapping is well-defined. 
	As $\mathcal{A} \subset \hat{\mathcal{A}}$, we find that any $\mathcal{A}$-measurable $f:X \rightarrow \reals$ is $\hat{\mathcal{A}}$-measurable, which justifies writing $[f]_{\hat{\sim}}$. 
	To see that the choice of representative of $[f]_\sim$ is irrelevant, let $f$ and $g$ be $\mathcal{A}$-measurable functions with $[f]_\sim = [g]_\sim$. Then $[f-g]_\sim = 0$, hence there is a $\nu$-null set $N$ such that $f(x) - g(x) = 0$ for all $x \in X \setminus N$. But by definition of the completion, $N$ is a $\hat{\nu}$-zero set as well, which shows that $[f-g]_{\hat{\sim}}=0$. Hence we get $[f]_{\hat{\sim}} = [g]_{\hat{\sim}}$. 
	
	The linearity of the mapping is obvious. 
	Moreover, for an $\mathcal{A}$-measurable $f:X \rightarrow \reals$ with $[f]_\sim \in \gL^\infty (\nu)$, we have 
	\begin{align}\Vert [f]_{\sim} \Vert_{\gL^\infty(\nu)}
		&= \inf \{B\geq 0 \, | \, \nu (\{x \in X \, | \, f(x)>B\}) = 0\} \\
		&= \inf \{B\geq 0 \, | \, \hat{\nu} (\{x \in X \, | \, f(x)>B\}) = 0\} \\
		&= \Vert [f]_{\hat{\sim}} \Vert_{\gL^\infty(\hat{\nu})} \text{,}
	\end{align}
	which shows that the operator is isometric and hence injective. 
	As any $[f]_{\hat{\sim}} \in \gL^\infty (\hat{\nu})$ can be reached by mapping the element $[\mathbbm{1}_{X \setminus N} f]_\sim$, surjectivity follows as well. 
\end{proof}

The following lemma gives some details on the construction of $\overline{H}_{\nu}^{(\theta,r)}$ in Theorem \ref{thm interp space BSF}. 
\begin{lemma} \label{lemma interp space BSF}
	Let $(X,\mathcal{A},\nu)$, $H$, $\mu_i$, and $e_i$ be as in Assumption \refirgendwo{assumption RKHS strong}{H+}\anmerkung{$\sigma$-endlich und $\nu$-vollständig, da Lifting verwendet wird}. 
	Moreover, let $0<\theta<1$ and $1\leq r \leq \infty$ be such that the embedding $[H]_\sim^{(\theta,r)} \hookrightarrow \gL^\infty (\nu)$ exists and is continuous. 
	Then for all $i \in I$, there exist representatives $\overline{e}_i \in \mathcal{L}^\infty (X)$ of $[e_i]_\sim$ such that the following statements are true: 
	\begin{enumerate}
		\item The series $\sum_{i \in I} b_i \overline{e}_i $ converges uniformly for all $b \in \ell^{(\mu,\theta,r)}$. 
		\item 
		There is a $\nu$-null set $N \subset X$ such that $\sum_{i \in I} b_i \overline{e}_i (x) = \sum_{i \in I} b_i e_i (x)$ for all $b \in \ell^{(\mu,\theta,r)}$ and all $x \in X \setminus N$. 
		\item The space 
		\begin{align}
			\overline{H}_{\nu}^{(\theta,r)}
			:= \left\{ \sum_{i \in I} b_i \overline{e}_i \, \middle| \, b \in \ell^{(\mu,\theta,r)} \right\}
		\end{align}
		is a BSF, where the BSF-norm is inherited from the $\ell^{(\mu,\theta,r)}$-norm of the Fourier coefficients $b_i$. 
		\item The mapping 
		\begin{align}
			[H]_\sim^{(\theta,r)} \rightarrow \overline{H}_{\nu}^{(\theta,r)}, \quad 
			\sum_{i \in I} b_i [e_i]_\sim \mapsto \sum_{i \in I} b_i \overline{e}_i
		\end{align}
		is an isometric isomorphism. 
		\item It holds that  
		\begin{align}
			\Vert \overline{H}_{\nu}^{(\theta,r)} \hookrightarrow \mathcal{L}^\infty (X) \Vert 
			= \Vert [H]_\sim^{(\theta,r)} \hookrightarrow \gL^\infty (\nu) \Vert \text{.}
		\end{align}
	\end{enumerate}		
\end{lemma}
\begin{proof}
	Let us begin with some preparations. To this end, let $(X, \hat{\mathcal{A}}, \hat{\nu})$ be the completion of the measure space $(X, \mathcal{A}, \nu)$, and let $[\,\cdot \,]_{\hat{\sim}}$ denote equivalence classes with respect to $\hat{\nu}$. 
	As the RKHS $H$ has a bounded kernel, the eigenfunctions $e_i$ are bounded as well. This gives $e_i \in \mathcal{L}^\infty(\nu) \subset \mathcal{L}^\infty(\hat{\nu})$. 
	Now by Theorem \ref{thm lifting}, the lifted eigenfunctions $\breve{e}_i := \phi([e_i]_{\hat{\sim}})$
	are $\hat{\mathcal{A}}$-measurable, bounded, and we know that for each $i \in I$, there is a $\hat{\nu}$-null set $N'_i$ such that $e_i(x)=\breve{e}_i(x)$ for all $x \in X \setminus N'_i$. As $I$ is at most countable,  $\bigcup_{i \in I} N'_i$ is again a $\hat{\nu}$-null set. By definition of the  completion, there hence is a $\nu$-null set $N'$ with $\bigcup_{i \in I} N'_i \subset N'$. 
	Additionally, by Lemma \ref{lemma completion}, there are $\nu$-null sets $N''_i$ such that $\mathbbm{1}_{X \setminus N''_i} \breve{e}_i$ is $\mathcal{A}$-measurable. 
	
	Let us consider the $\nu$-null set $N:= N' \cup \bigcup_{i \in I} N''_i$. Then for all $i \in I$, it holds that $\overline{e}_i := \mathbbm{1}_{X \setminus N} \breve{e}_i$ is $\mathcal{A}$-measurable with $\overline{e}_i (x) = e_i (x)$ for all $x \in X \setminus N$. 
	
	Let us denote the linear operator 
	\begin{align}
		A: \gL^{\infty} (\nu) \rightarrow \mathcal{L}^\infty (X), \quad [f]_\sim \mapsto \mathbbm{1}_{X \setminus N} \phi([f]_{\hat{\sim}})
	\end{align}
	which is well-defined by Lemma \ref{lemma completion} and has operator norm at most $1$ due to Lemma \ref{lemma completion} and Theorem \ref{thm lifting}. Note that $A [e_i]_\sim =\overline{e}_i$ are $\mathcal{A}$-measurable functions on $X$. 
	
	Furthermore, we write 
	\begin{align}
		\kappa:=\Vert [H]_\sim^{(\theta,r)} \hookrightarrow \gL^\infty(\nu) \Vert \text{.} 
	\end{align}
	
	\ada{i} Let $b \in \ell^{(\mu,\theta,r)}$. 
	Then by Lemma \ref{lemma convergence in interp norm}, we know that  $\sum_{i \in I} b_i [e_i]_{\sim}$ converges in $[H]_\sim^{(\theta,r)}$. 
	Due to the continuity of the embedding $[H]_\sim^{(\theta,r)} \hookrightarrow \gL^\infty (\nu)$, it also converges in $\gL^\infty (\nu)$, and we have 
	\begin{align} \label{formula l infty inequality 1}
		\left\Vert \sum_{i \in I} b_i [e_i]_{\sim} \right\Vert_{\gL^{\infty}(\nu)}
		\leq \kappa \left\Vert \sum_{i \in I} b_i [e_i]_{\sim} \right\Vert_{[H]_\sim^{(\theta,r)}}
		=  \kappa \Vert b \Vert_{(\mu, \theta, r)} \text{.}
	\end{align}
	By continuity of the operator $A: \gL^{\infty} (\nu) \rightarrow \mathcal{L}^\infty (X)$, we thus find 
	\begin{align} \label{formula lifting of series}
		A \left( \sum_{i \in I} b_i [e_i]_{\sim} \right) 
		= \sum_{i \in I} b_i A ([e_i]_{\sim})
		= \sum_{i \in I} b_i \overline{e}_i \text{,}
	\end{align}
	so $\sum_{i \in I} b_i \overline{e}_i$ converges in $\mathcal{L}^\infty(X)$, that is uniformly. 
	
	\ada{ii} We know that $\overline{e}_i (x) = e_i(x)$ holds for all $i \in I$ and all $x \in X \setminus N$, which yields the assumption by (i). 
	
	\ada{iii} As $\ell^{(\mu,\theta,r)}$ is a Banach space by Remark \ref{remark isometric isomorphisms}, it follows that $\overline{H}_{\nu}^{(\theta,r)}$ is a Banach space as well. 
	By the uniform convergence in (i), $\overline{H}_{\nu}^{(\theta,r)}$ consists of functions. 
	It remains to show that the evaluation functionals are  continuous. 
	To this end, let us fix a $b \in \ell^{(\mu,\theta,r)}$. Because $A$ has operator norm at most $1$, we find by \eqref{formula lifting of series} that 
	\begin{align}
		\left\Vert \sum_{i \in I} b_i \overline{e}_i \right\Vert_{\mathcal{L}^{\infty}(X)}
		= \left\Vert A \left( \sum_{i \in I} b_i [e_i]_{\sim} \right) \right\Vert_{\mathcal{L}^{\infty}(X)}
		\leq \left\Vert \sum_{i \in I} b_i [e_i]_{\sim} \right\Vert_{\gL^{\infty}(\nu)} \text{.} \label{formula l infty inequality 2} 
	\end{align}
	Now let $\delta_x: \mathcal{L}^\infty(X) \rightarrow \reals, f \mapsto f(x)$ be the evaluation functional at some fixed $x \in X$. 
	Combining \eqref{formula l infty inequality 1} and \eqref{formula l infty inequality 2}, we find 
	\begin{align} \label{formula evaluation functionals bounded}
		\left| \delta_x \left(\sum_{i \in I} b_i \overline{e}_i\right) \right|
		= \left| \left(\sum_{i \in I} b_i \overline{e}_i\right) (x) \right|
		\leq 
		\left\Vert \sum_{i \in I} b_i \overline{e}_i \right\Vert_{\mathcal{L}^{\infty}(X)}
		\leq \kappa \Vert b \Vert_{(\mu, \theta, r)} \text{.}
	\end{align}
	\anmerkung{Banachraum ist es weil $\ell^{(\mu,\theta,r)}$ einer ist, siehe (iv)}
	
	\ada{iv} This follows as the norms on $[H]_\sim^{(\theta,r)}$ and on $\overline{H}_{\nu}^{(\theta,r)}$ are both inherited from the $\ell^{(\mu,\theta,r)}$-norm of the unique\anmerkung{als eind $\gL^2$-Fourierkoeff.} Fourier coefficients. 
	
	\ada{v} Note that by using \eqref{formula l infty inequality 1} and \eqref{formula l infty inequality 2}, we find 
	\begin{align}
		\left\Vert \sum_{i \in I} b_i \overline{e}_i \right\Vert_{\mathcal{L}^{\infty}(X)}
		\leq \kappa \Vert b \Vert_{(\mu, \theta, r)} 
		= \kappa \left\Vert \sum_{i \in I} b_i e_i \right\Vert_{[H]_\sim^{(\theta,r)}} \text{,}
	\end{align}
	which shows \enquote{$\leq$}. For \enquote{$\geq$}, we use 
	\begin{align}
		\Vert \overline{H}_{\nu}^{(\theta,r)} \hookrightarrow \mathcal{L}^\infty (X) \Vert 
		= \sup_{\Vert f \Vert_{\overline{H}_{\nu}^{(\theta,r)}} \leq 1} \Vert f \Vert_{\mathcal{L}^\infty (X)}
		\geq \sup_{\Vert f \Vert_{\overline{H}_{\nu}^{(\theta,r)}} \leq 1} \Vert [f]_\sim \Vert_{\gL^\infty (\nu)}
		= \Vert [H]_\sim^{(\theta,r)} \hookrightarrow \gL^\infty (\nu) \Vert
	\end{align}
	where the last step follows by $\Vert f \Vert_{\overline{H}_{\nu}^{(\theta,r)}} = \Vert [f]_\sim \Vert_{[H]_\sim^{(\theta,r)}}$. 
\end{proof}

\begin{proof}[Proof of Theorem \ref{thm interp space BSF}]\labelirgendwo{proof thm interp space BSF} 
	The assertion follows directly from Lemma \ref{lemma interp space BSF} by using the inverse of the isometric isomorphism from (iv). 
	The additional statement follows from (ii). 
\end{proof}


In Theorem \ref{thm interp space BSF}, we describe by means of lifting theory how to find representatives of $[H]_\sim^{(\theta,r)}$ that give rise to a Banach space of functions. However, note that we already chose a set of \enquote{canonical} representatives in Definition \ref{def power space}. Let us analyse under which circumstances this set of \enquote{canonical} representatives forms a BSF. 
To this end, let us apply a topological technique similar to \cite{KLE}.

\begin{proof}[Proof of Theorem \ref{thm interp space BSF with topology}] \labelirgendwo{proof thm interp space BSF with topology}
	Let us denote 
	\begin{align}
		\kappa:=\Vert [H]_\sim^{(\theta,r)} \hookrightarrow \gL^\infty(\nu) \Vert \text{.} 
	\end{align}
	
	\ada{i} Let $b \in \ell^{(\mu,\theta,r)}$. 
	Then by Lemma \ref{lemma convergence in interp norm}, we know that  $\sum_{i \in I} b_i [e_i]_{\sim}$ converges in $[H]_\sim^{(\theta,r)}$. 
	Due to the continuity of the embedding $[H]_\sim^{(\theta,r)} \hookrightarrow \gL^\infty (\nu)$, it converges in $\gL^\infty (\nu)$ as well, and we have 
	\begin{align} \label{formula l infty inequality 1 cont}
		\left\Vert \sum_{i \in I} b_i [e_i]_{\sim} \right\Vert_{\gL^{\infty}(\nu)}
		\leq \kappa \left\Vert \sum_{i \in I} b_i [e_i]_{\sim} \right\Vert_{[H]_\sim^{(\theta,r)}}
		=  \kappa \Vert b \Vert_{(\mu, \theta, r)} \text{.}
	\end{align}
	Now the partial sums of $\sum_{i \in I} b_i e_i$ are $\tau(H)$-continuous and $\sum_{i \in I} b_i [e_i]_\sim = [ \sum_{i \in I} b_i e_i]_\sim$ converges in $\gL^\infty (\nu)$-norm. 
	By Lemma \ref{lemma Cb closed}, we hence find that $\sum_{i \in I} b_i e_i$ converges in $C_b(X, \tau(H))$, hence uniformly. 
	
	\ada{ii} First of all, note that the coefficients $b_i$ are uniquely determined by the function $\sum_{i \in I} b_i e_i$, as the $[e_i]_\sim$ are $\gL^2(\nu)$-orthogonal. 
	
	As $\ell^{(\mu,\theta,r)}$ is a Banach space by Remark \ref{remark isometric isomorphisms}, it follows that ${H}_{\nu}^{(\theta,r)}$ is a Banach space as well. 
	By the uniform convergence in (i), ${H}_{\nu}^{(\theta,r)}$ consists of functions. 
	It remains to show that the evaluation functionals are  continuous. 
	To this end, let us fix a $b \in \ell^{(\mu,\theta,r)}$. 
	Then for all $x \in X$, we consider $\delta_x: \mathcal{L}^\infty(X) \rightarrow \reals, f \mapsto f(x)$ to find that 
	\begin{align} \label{formula evaluation functionals bounded cont}
		\begin{split}
			\left| \delta_x \left(\sum_{i \in I} b_i e_i\right) \right|
		= \left| \left(\sum_{i \in I} b_i e_i\right) (x) \right|
		\leq 
		\left\Vert \sum_{i \in I} b_i e_i \right\Vert_{\mathcal{L}^{\infty}(X)}
		= \left\Vert \sum_{i \in I} b_i [e_i]_\sim \right\Vert_{\gL^{\infty}(\nu)}
		\leq \kappa \Vert b \Vert_{(\mu, \theta, r)} \text{,}
		\end{split}
	\end{align}
	where for second to last step, we used Corollary \ref{corollary L infty = L infty}, and the last inequality is due to \eqref{formula l infty inequality 1 cont}. 
	
	\ada{iii} This follows as the norms on $[H]_\sim^{(\theta,r)}$ and on ${H}_{\nu}^{(\theta,r)}$ are both inherited from the $\ell^{(\mu,\theta,r)}$-norm of the unique\anmerkung{als eind $\gL^2$-Fourierkoeff.} Fourier coefficients. 
	
	\ada{iv} In the proof of (i), we found that ${H}_{\nu}^{(\theta,r)}$ consists of $\tau(H)$-continuous and bounded functions. 
%

	Note that by using the isometric isomorphism $\alpha: {H}_{\nu}^{(\theta,r)} \rightarrow [H]_\sim^{(\theta,r)}$ from (iii) and the isometry $\beta: \gC_b(X, \tau(H)) \rightarrow \gL^\infty (\nu), f \mapsto [f]_\sim$, the following diagram commutes: 
	\begin{center}
		\begin{tikzcd}
			{{H}_{\nu}^{(\theta,r)}} \arrow[hookrightarrow,rr,"\Id"] \arrow[dd, "\simeq","\alpha"'] &  & {\gC_b(X, \tau(H))} \arrow[dd,"\beta"]                 \\
			&  &                    \\
			{[H]_\sim^{(\theta,r)}} \arrow[hookrightarrow,rr,"\Id"]                 &  & {\gL^\infty (\nu)}
		\end{tikzcd}
	\end{center}
	But as $\Id: [H]_\sim^{(\theta,r)} \rightarrow \gL^\infty(\nu)$ is bounded with operator norm $\kappa$, it follows that $\Id: {H}_{\nu}^{(\theta,r)} \rightarrow \gC_b(X, \tau(H))$ is bounded with operator norm $\kappa$ as well.

	\ada{v} Let $f \in H$. Then by Lemma \ref{lemma H sequence space unitary operator}, there is a sequence $(b_i)_{i \in I} \in \ell^2(\mu^{-1})$ such that $[f]_\sim=\sum_{i \in I} b_i [e_i]_\sim$. It holds that  
	\begin{align}
		\ell^2(\mu^{-1}) \hookrightarrow [\ell^2(I), \ell^2(\mu^{-1})]_{(\theta,r)} \triangleq \ell^{(\mu,\theta,r)} \text{,}
	\end{align}
	where for the first step, we apply Lemma \ref{lemma interpolation space subsets} and for the second, we apply Lemma \ref{lemma weighted Lp interpolation} to the counting measure on $I$ and the weight sequence $\mu^{-1}$ to find that the spaces $[\ell^2(I), \ell^2(\mu^{-1})]_{(\theta,r)}$ and $\ell^{(\mu,\theta,r)}$ are norm equivalent. 
	
	We hence find that $(b_i)_{i \in I} \in \ell^{(\mu,\theta,r)}$. Now by (i), the sum $\sum_{i \in I} b_i e_i$ converges uniformly in $C_b(X,\tau(H))$. Because $f$ itself is $\tau(H)$-continuous as well, we have $f=\sum_{i \in I} b_i e_i$, as almost everywhere agreeing $\tau(H)$-continuous functions agree everywhere by Corollary \ref{corollary L infty = L infty}. This shows that $f=\sum_{i \in I} b_i e_i \in H^{(\theta,r)}$. 
\end{proof}


\begin{proof}[Proof of Theorem \ref{thm embedding equivalence}] \labelirgendwo{proof thm embedding equivalence}
	First of all, note that by Definition \ref{def power space}, an arbitrary $[g]_{\sim} \in [H]_\sim^{(\theta,r)}$ can be written as a series 
	\begin{align} \label{formula series representation of g}
		[g]_{\sim} = \sum_{i \in I} b_i [e_i]_{\sim}
	\end{align}
	for a suitable $b \in \ell^{(\mu, \theta, r)}$, where the convergence is in $[H]_\sim^{(\theta,r)}$ by Lemma \ref{lemma convergence in interp norm}. In addition, it holds that 
	\begin{align}
		\Vert [g]_{\sim} \Vert_{[H]_\sim^{(\theta,r)}}
		= \Vert b \Vert_{(\mu,\theta,r)} \text{.}
	\end{align}
	For the rest of this proof, we write $f_i(x) := e_i(x) \mu_i^{\theta}$ for all $i \in I$ and all $x \in X$. 
	
	Concerning the different assumptions, note that Assumptions \refirgendwo{assumption RKHS strong}{H+}  and \refirgendwo{assumption T}{T} imply Assumption \refirgendwo{assumption RKHS}{H}. 
	
	\enquote{(i) $\Rightarrow$ (ii)} Obvious. 
	
	\enquote{(ii) $\Rightarrow$ (iii)}
	By Property (ii), there is a $\nu$-null set $N$ such that 
	\begin{align}
		\left\Vert \left(f_i(x)\right)_{i \in I}\right\Vert_{(\mu,\theta,r')} 
		\leq \kappa
	\end{align}
	holds for all $x \in X \setminus N$. 		
	Hence it follows for $x \in X \setminus N$ and all $b \in \ell^{(\mu,\theta,r)}$ that 
	\begin{align}
		\begin{split}
			\sum_{i \in I} | b_i e_i(x) |
		&= \sum_{i \in I} | b_i f_i(x) | \mu_i^{-\theta} \\
		&\leq  \Vert (f_i(x))_{i \in I} \Vert_{(\mu,\theta,r')} \Vert b \Vert_{(\mu,\theta,r)} \\
		&\leq \kappa \,  \Vert b \Vert_{(\mu,\theta,r)} \text{,}  \label{formula sum bounded outside N}
		\end{split}
	\end{align}
	where the second to last inequality is due to Lemma \ref{lemma Hölder for ell theta r}. 
	Hence for infinite $I$, $s_n:=\sum_{i = 1}^n b_i e_i(x)$ converges for all $x \in X \setminus
	 N$ where $n \rightarrow \infty$. 
	In addition, we know that this sum converges to $[g]_{\sim}$ in $[H]_\sim^{(\theta,r)}$, hence it converges in $\gL^2(\nu)$ as we have $[H]_\sim^{(\theta,r)} \hookrightarrow \gL^2(\nu)$ by Remark \ref{remark convergence in L2}. This $\gL^2(\nu)$-convergence implies that a subsequence $s_{n_k} = \sum_{i=1}^{n_k} b_i e_i (x)$ converges to $g(x)$ for all $x \in X \setminus \widetilde{N}$, where $\tilde{N}$ is another $\nu$-zero set. 
	Together with the mentioned almost everywhere convergence of $s_n$, this yields that $g(x) = \sum_{i \in I} b_i e_i(x)$ is true for all $x \in X \setminus (N \cup \widetilde{N})$.  
	Consequently, we find 
	\begin{align}
		\Vert [g]_{\sim} \Vert_{\gL^{\infty}(\nu)}
		&= \inf \left\{ B \geq 0 \, : \, \exists M \text{ with } \nu(M)=0 \text{ and }\sup_{x \in X \setminus M}\left|g(x) \right|\leq B \right\} \\
		&\leq \sup_{x \in X \setminus (N \cup \tilde{N})} \left| g(x) \right| \\
		&= \sup_{x \in X \setminus (N \cup \tilde{N})} \left| \sum_{i \in I} b_i e_i (x) \right| \\
		& \leq \kappa \Vert b \Vert_{(\mu,\theta,r)} \\
		&= \kappa \Vert [g]_{\sim} \Vert_{[H]_\sim^{(\theta,r)}}
	\end{align}
	by an application of \eqref{formula sum bounded outside N}. This proves (iii). 
	
	\enquote{(iii) $\Rightarrow$ (ii)} under Assumption \refirgendwo{assumption RKHS strong}{H+}: 
	By Lemma \ref{lemma interp space BSF}, for almost all $x \in X$ and all $b \in \ell^{(\mu,\theta,r)}$, we have 
	\begin{align}
		\left| \sum_{i \in I} b_i f_i(x) \mu_i^{-\theta} \right|
		= \left| \sum_{i \in I} b_i e_i(x) \right| 
		= \left| \sum_{i \in I} b_i \overline{e}_i(x) \right|
		\leq \kappa \left\Vert \sum_{i \in I} b_i \overline{e}_i \right\Vert_{\overline{H}_{\nu}^{(\theta,r)}}
		= \kappa \Vert b \Vert_{(\mu, \theta, r)} \text{, }
	\end{align}
	where $\overline{e}_i$ is defined as in Lemma \ref{lemma interp space BSF}. Note that the sums are absolutely convergent by Remark \ref{remark absoulte values in sums}.  
	For almost all $x \in X$, an application of Lemma \ref{lemma finiteness if all dual pairings finite} to $a_i := f_i(x)$ yields that $(f_i(x))_{i \in I} \in \ell^{(\mu,\theta,r')}$ with norm $\Vert f_i(x) \Vert_{(\mu,\theta,r')} \leq \kappa$. 
	This shows (ii). 
	
	\enquote{(ii) $\Rightarrow$ (i)} under Assumption \refirgendwo{assumption T}{T}: 
	An application of Lemma \ref{lemma non-negative lower semicontinuous} shows that the function $X \rightarrow [0,\infty]$ given by 
	\begin{align}
		x \mapsto \left\Vert \left( e_i(x) \mu_i^{\theta} \right)_{i \in I}\right\Vert_{(\mu,\theta,r')} 
	\end{align}
	is non-negative and lower semicontinuous. Hence by Lemma \ref{lemma sup is esssup}, the essential supremum and the supremum agree. 
	
	\enquote{(iii) $\Rightarrow$ (ii)} under Assumption \refirgendwo{assumption T}{T}: 
	By Theorem \ref{thm interp space BSF with topology}, for all $x \in X$ and all $b \in \ell^{(\mu,\theta,r)}$, we have 
	\begin{align}
		\left| \sum_{i \in I} b_i f_i(x) \mu_i^{-\theta} \right|
		= \left| \sum_{i \in I} b_i e_i(x) \right| 
		\leq \kappa \left\Vert \sum_{i \in I} b_i e_i \right\Vert_{H_\nu^{(\theta,r)}}
		= \kappa \Vert b \Vert_{(\mu, \theta, r)} \text{.}
	\end{align}
	Thus, for all $x \in X$, another application of Lemma \ref{lemma finiteness if all dual pairings finite} to $a_i := f_i(x)$ yields that $(f_i(x))_{i \in I} \in \ell^{(\mu,\theta,r')}$ with norm $\Vert f_i(x) \Vert_{(\mu,\theta,r')} \leq \kappa$. 
	This shows (ii). 	
\end{proof}
\review
{
\begin{proof}[Proof of Corollary \ref{corollary power space reiteration}]
	By Theorem \ref{thm equivalent norm}, we know that $\Id: [H]_\sim^{(\theta_i,r_i)} \rightarrow [\gL^2(\nu),[H]_\sim]_{(\theta_i,r_i)}$ is an isometric isomorphism for $i=1,2$. 
	By an application of Lemma \ref{lemma interpolation space isometric isomorphism}, we find 
	\begin{align}
		\left[[H]_\sim^{(\theta_1,r_1)},[H]_\sim^{(\theta_2,r_2)}\right]_{(\theta,r)}
		\triangleq \left[[\gL^2(\nu),[H]_\sim]_{(\theta_1,r_1)},[\gL^2(\nu),[H]_\sim]_{(\theta_2,r_2)}\right]_{(\theta,r)}
		\triangleq [\gL^2(\nu),[H]_\sim]_{(\eta,r)}
	\end{align}
	where $\eta=(1-\theta)\theta_1 + \theta \theta_2$. 
	For the last step, we used the reiteration theorem, see e.g. \cite[Chapter 5 Theorem 2.4]{bennet_sharpley}. \anmerkung{Zusammen mit \cite[Chapter 5 Corollary 1.8]{bennet_sharpley}} The result follows by another application of Theorem \ref{thm equivalent norm}. 
\end{proof}
}

\todo[inline,color=cyan]{Bis hier kleinere Veränderungen}

\subsection{Proofs concerning applications}

\begin{lemma} \label{lemma estimate theta beta norms}
	Suppose that $(\mu_i)_{i \in I}$ is a non-increasing sequence of positive numbers, $0 < \beta < 1$, $0< \theta < 1$ and $1\leq r \leq \infty$. Let $(\alpha_i)_{i \in I}$ and $(b_i)_{i \in I}$ be sequences of real numbers. 
	Then it holds that 
	\begin{align}
		\Vert (\alpha_i b_i)_{i \in I} \Vert_{(\mu,\theta,r)} 
		\leq \sup_{i \in I} \left( | \alpha_i | \mu_i^{\frac{\beta-\theta}{2}}\right) \Vert b \Vert_{(\mu,\beta,r)} \text{.}
	\end{align}
\end{lemma}
\begin{proof}
	
	\begin{align}
		\Vert (\alpha_i b_i)_{i \in I} \Vert_{(\mu,\theta,r)}
		&= \left( \sum_{j \in \integers} \left( \sum_{i \in M_j} \big(\alpha_i b_i\big)^2 \mu_i^{-\theta} \right)^{\frac{r}{2}} \right)^{\frac{1}{r}} \\
		&= \left( \sum_{j \in \integers} \left( \sum_{i \in M_j} \alpha_i^2 \mu_i^{\beta-\theta} b_i^2 \mu_i^{-\beta} \right)^{\frac{r}{2}} \right)^{\frac{1}{r}} \\
		&\leq \sup_{i \in I} \left( |\alpha_i |\mu_i^{\frac{\beta-\theta}{2}} \right)   \left( \sum_{j \in \integers} \left( \sum_{i \in M_j}  b_i^2 \mu_i^{-\beta} \right)^{\frac{r}{2}} \right)^{\frac{1}{r}} \\
		&= \sup_{i \in I} \left( |\alpha_i |\mu_i^{\frac{\beta-\theta}{2}} \right) \Vert b \Vert_{(\mu,\beta,r)}
	\end{align}
\end{proof}

\begin{proof}[Proof of Proposition \ref{prop regularisation error in interp norm}] \labelirgendwo{proof prop regularisation error in interp norm}
	For $i \in I$, we write $b_i := \langle f_{L,P}^* , [e_i]_{\sim} \rangle_{\gL^2(P_X)}$. 
	Then we find that 
	\begin{align} \label{formula series f*}
		f_{L,P}^* = \sum_{i \in I} b_i [e_i]_{\sim}
	\end{align}
	with convergence in $\gL^2(P_X)$ as $f^*_{L,P} \in \gL^2(P_X)$ by Remark \ref{remark convergence in L2}. 
	In addition, Theorem \ref{thm equivalent norm} yields a $C_4 \geq 0$\anmerkung{Problem: $C_4$ kann von $\mu$ abhängen, siehe \ref{thm equivalent norm} (das steht bei Ingo als unabhängig)} independent of $f_{L,P}^*$ such that 
	\begin{align} \label{formula upper bound on b}
		\Vert b \Vert_{(\mu,\beta,r)} 
		\leq C_4 \Vert f_{L,P}^* \Vert_{[\gL^2(X),[H]_{\sim}]_{\beta,r}} \text{.}
	\end{align}
	Now note that we have $(\mu_i g_\lambda (\mu_i) b_i)_{i \in I} \in \ell^2(\mu^{-1})$ because $(b_i)_{i \in I}$ is an $\ell^2$-sequence and $g_\lambda$ is bounded by Definition \ref{def admissible filter function}. 
	Hence by \eqref{formula spectral regularisation}, we see that 
	\begin{align}
		f_{P,\lambda,g}
		= \sum_{i \in I} {\mu_i g_\lambda (\mu_i) b_i} e_i
	\end{align}
	converges in $H$. 
	Hence it follows that 	
	\begin{align} \label{formula orthogonal series f_P,lambda,g}
		[f_{P,\lambda,g}]_{\sim}
		= I_k f_{P,\lambda,g}
		= I_k \left( \sum_{i \in I} \mu_i g_\lambda (\mu_i) b_i e_i\right)
		= \sum_{i \in I} \mu_i g_\lambda (\mu_i) b_i [e_i]_{\sim}
	\end{align}
	with convergence in $\gL^2(P_X)$. 
	
	\ada{i} By the $\gL^2(P_X)$-orthogonal series of $f_{L,P}^*$ in \eqref{formula series f*} and the one of $[f_{P,\lambda,g}]_{\sim}$ in \eqref{formula orthogonal series f_P,lambda,g}, we conclude 
	\begin{align}
		[f_{P, \lambda, g}]_{\sim} - f_{L,P}^*
		= \sum_{i \in I} (1-\mu_i g_\lambda (\mu_i)) b_i [e_i]_{\sim}
	\end{align}
	with convergence in $\gL^2(P_X)$. 
	
	Then by the norm equivalence from Theorem \ref{thm equivalent norm}, we find a $C_5 \geq 0$\anmerkung{Analoges Problem mit $C_5$} independent of $f_{L,P}^*$ such that we can estimate as follows via Lemma \ref{lemma estimate theta beta norms}: 
	\begin{align}
		\Vert [f_{P, \lambda, g}]_{\sim} - f_{L,P}^* \Vert_{[\gL^2(X),[H]_{\sim}]_{\theta,r}}
		&\leq C_5 \Vert (1-\mu_i g_\lambda (\mu_i)) b_i \Vert_{(\mu,\theta,r)} \\
		&\leq C_5 \sup_{i \in I} \left( \big|1-\mu_i g_\lambda (\mu_i)\big|\mu_i^{\frac{\beta-\theta}{2}} \right) \Vert b \Vert_{(\mu,\beta,r)}
	\end{align}
	From this, the assertion follows by the upper bound on $\Vert b \Vert_{(\mu,\beta,r)}$ given in \eqref{formula upper bound on b}. 
	
	\ada{ii} By the norm equivalence in Theorem \ref{thm equivalent norm}, we find a $C_6 \geq 0$\anmerkung{Analoges Problem mit $C_6$} independent of $f_{L,P}^*$ such that we can estimate as follows via Lemma \ref{lemma estimate theta beta norms}: 
	\begin{align}
		\Vert [f_{P, \lambda, g}]_{\sim} \Vert_{[\gL^2(X),[H]_{\sim}]_{\theta,r}}
		&\leq C_6 \Vert \mu_i g_\lambda (\mu_i) b_i \Vert_{(\mu,\theta,r)} \\
		&\leq C_6 \sup_{i \in I} \left( |\mu_i g_\lambda (\mu_i)|\mu_i^{\frac{\beta-\theta}{2}} \right) \Vert b \Vert_{(\mu,\beta,r)}
	\end{align}
	From this, the estimate follows by the upper bound on $\Vert b \Vert_{(\mu,\beta,r)}$ given in \eqref{formula upper bound on b}. 
\end{proof}
	
	\newcounter{appendix}
	\setcounter{appendix}{0}
	\renewcommand\thesection{\Alph{appendix}}
	\stepcounter{appendix}
\section{Appendix: Interpolation spaces}

In this appendix, we collect some auxiliary results on interpolation spaces, which are probably known to experts. 

We begin by proving that isometric isomorphisms can be carried over from Banach spaces $E_2 \hookrightarrow E_1$ to the interpolation spaces $[E_1,E_2]_{\theta,r}$. 
For a map $f: A \rightarrow B$ and $A' \subset A$, $B'  \subset B$, we use the notation $f |_{A'\rightarrow B'}$ for the map $A' \rightarrow B'$ given by $x \mapsto f(x)$ in case it is well-defined. 
\begin{lemma} \label{lemma interpolation space isometric isomorphism}\anmerkung{Kann man auch allgemeiner machen: $E_2 \not\hookrightarrow E_1$}
	Let $E_2 \hookrightarrow E_1$ and $F_2 \hookrightarrow F_1$ be Banach spaces. Suppose there is a bounded linear operator $\phi: E_1 \rightarrow F_1$ such that $\phi$ and $\phi |_{E_2 \rightarrow F_2}$ are isometric isomorphisms. Then for $0<\theta<1$ and $1\leq r \leq \infty$, the following map is an isometric isomorphism: 
	\begin{align}
		\phi |_{[E_1,E_2]_{\theta,r} \rightarrow [F_1,F_2]_{\theta,r}}
	\end{align}
\end{lemma}
\begin{proof} \anmerkung{Insbesondere für allgemeineren Fall nutze den alten abstrakten Beweis}
	Let us begin by showing that $\phi |_{[E_1,E_2]_{\theta,r} \rightarrow [F_1,F_2]_{\theta,r}}$ is an isometry. To this end note that for all $x \in E_1$ and $t > 0$ we have 
	\begin{align}
		K(x,t,E_1,E_2) 
		&= \inf_{ y \in E_2 } \left( \Vert x-y \Vert_{E_1} + t \Vert y \Vert_{E_2} \right) \\
		&= \inf_{ y \in E_2 } \left( \Vert \phi (x-y) \Vert_{F_1} + t \Vert \phi (y) \Vert_{F_2} \right) \\
		&= \inf_{ z \in F_2 } \left( \Vert \phi (x) - z \Vert_{F_1} + t \Vert z \Vert_{F_2} \right) \\
		&= K(\phi(x),t,F_1,F_2)
		\text{.}
	\end{align}
	This shows that $\Vert \phi (x) \Vert_{[F_1,F_2]_{\theta,r}} = \Vert x \Vert_{[E_1,E_2]_{\theta,r}}$, that is $\phi |_{[E_1,E_2]_{\theta,r} \rightarrow [F_1,F_2]_{\theta,r}}$ is isometric. 
	Finally, applying the same argument to $\phi^{-1}_{E_i \rightarrow F_i}$ shows that $\phi^{-1} |_{[F_1,F_2]_{\theta,r} \rightarrow [E_1,E_2]_{\theta,r}}$ is also isometric. Obviously, it is the inverse to $\phi |_{[E_1,E_2]_{\theta,r} \rightarrow [F_1,F_2]_{\theta,r}}$. 
\end{proof}

Let us collect some properties of interpolation spaces for the situation considered here, that is $E_2 \hookrightarrow E_1$: 
\begin{lemma} \label{lemma interpolation space subsets}
	Let $E_2 \hookrightarrow E_1$ be Banach spaces. Suppose that $0 < \theta_1 \leq \theta_2<1$ and $1\leq r_1 \leq r_2 \leq \infty$. Then the following properties hold: \anmerkung{Geht auch als gro\ss e Mengenungleichung}
	\begin{enumerate}
		\item $[E_1,E_2]_{\theta_1,r_1} \hookrightarrow E_1$
		\item $E_2 \hookrightarrow [E_1,E_2]_{\theta_1,r_1}$
		\item $ [E_1,E_2]_{\theta_2,r_1} \hookrightarrow [E_1,E_2]_{\theta_1,r_1}$
		\item $ [E_1,E_2]_{\theta_1,r_1} \hookrightarrow [E_1,E_2]_{\theta_1,r_2}$
	\end{enumerate}
\end{lemma}
\begin{proof}
	\ada{i} This follows by \cite[Proposition 1.8 in Chapter 5]{bennet_sharpley} because we have $E_1 + E_2 = E_1$ due to the embedding $E_2 \hookrightarrow E_1$. 
	
	\ada{ii} This follows by \cite[Proposition 1.8 in Chapter 5]{bennet_sharpley} because we have $E_1 \cap E_2 = E_2$ due to the embedding $E_2 \hookrightarrow E_1$. 
	
	\ada{iii} See \cite[3.4.1 (d)]{bergh_loefstroem}, where the notation is clarified in \cite[page VII]{bergh_loefstroem}\anmerkung{$\subset$ hei\ss t dort stetig eingebettet}. 
	
	\ada{iv} This follows by \cite[Proposition 1.10 in Chapter 5]{bennet_sharpley}. 
\end{proof}

\begin{lemma} \label{lemma interpolation space direct sum}
	Let $H$ be a Hilbert space and let $H_1 \subset H$ be a closed subspace. 
	Suppose the Hilbert space $H_2$ satisfies $H_2 \hookrightarrow H_1$. Then for $0<\theta<1$ and $1\leq r \leq \infty$, it holds that 
	\begin{align}
		[H, H_2]_{\theta,r} = [H_1, H_2]_{\theta,r}
	\end{align}
	and the norms of both spaces are equal. 
\end{lemma}
\begin{proof} \anmerkung{Für $r<\infty$ geht der Beweis mit Dichtheit vom kleinen Raum im Interpolationsraum schneller. Die gilt aber für $r=\infty$ nicht, siehe \cite[Thm. 3.4.2]{bergh_loefstroem}. }
	Our first goal is to show 
	\begin{align} \label{formula union is in H1}
		[H, H_2]_{\theta,r} \cup [H_1, H_2]_{\theta,r} \subset H_1 \text{.}
	\end{align}
	To this end, we first note that $[H_1, H_2]_{\theta,r} \subset H_1$ by Lemma \ref{lemma interpolation space subsets}. 	
	
	For the proof of $[H, H_2]_{\theta,r} \subset H_1$ we pick an $x \in [H, H_2]_{\theta,r}$. As the latter space is included in $H = H_1 \oplus H_1^\perp$ by Lemma \ref{lemma interpolation space subsets}, we consider the unique decomposition $x=x_1+x_0$ with $x_1 \in H_1$ and $x_0 \in H_1^\perp$. 
	It suffices to show that $x_0=0$. To this end, we assume the converse, that is  $\Vert x_0 \Vert_{H} > 0$. Then for $t > 0$, it follows that 
	\begin{align}
		K(x,t,H, H_2)
		&= \inf_{y \in H_2} \left( \Vert x-y \Vert_{H} + t \Vert y \Vert_{H_2} \right) \\
		&=  \inf_{y \in H_2} \left( \Vert x_1-y+x_0 \Vert_{H} + t \Vert y \Vert_{H_2} \right) \\
		&=  \inf_{y \in H_2} \left( \sqrt{ \Vert x_1 - y \Vert_{H}^2 + \Vert x_0 \Vert_{H}^2} + t \Vert y \Vert_{H_2} \right) \\
		&\geq \Vert x_0 \Vert_H \text{.}
	\end{align}
	Hence for $1 \leq r < \infty$, we find that $x \notin [H, H_2]_{\theta,r}$ as 
	\begin{align}
		\Vert x \Vert_{[H, H_2]_{\theta,r}} 
		&= \left( \int_0^{\infty}  t^{-\theta r -1} K^r(x,t, H, H_2)  \gd t \right)^{\frac{1}{r}} \\
		&\geq \left( \int_0^{\infty} t^{-\theta r -1} \Vert x_0 \Vert_H^r  \gd t \right)^{\frac{1}{r}} \\
		&= \infty \text{.}
	\end{align}
	For $r=\infty$, we find the same result by estimating
	\begin{align}
		\Vert x \Vert_{[H, H_2]_{\theta,\infty}} 
		= \sup_{t>0} \left( t^{-\theta} K(x,t, H, H_2) \right)
		\geq \sup_{t>0} \left( t^{-\theta} \Vert x_0 \Vert_H \right)
		= \infty \text{.}
	\end{align}
	Hence we find a contradiction to $x \in [H, H_2]_{\theta,r}$. This shows the desired $x_0=0$. 
	
	Let us now fix an $x \in [H, H_2]_{\theta,r} \cup [H_1, H_2]_{\theta,r}$. 	
	By \eqref{formula union is in H1}, we note that $x \in H_1$ and hence we find 
	\begin{align}
		K(x,t,H, H_2)
		&= \inf_{y \in H_2} \left( \Vert x-y \Vert_{H} + t \Vert y \Vert_{H_2} \right) \\
		&= \inf_{y \in H_2} \left( \Vert x-y \Vert_{H_1} + t \Vert y \Vert_{H_2} \right) 
		= K(x,t,H_1, H_2) \text{.}
	\end{align}
	The equality of the interpolation spaces and their norms hence follows. 
\end{proof}

\stepcounter{appendix}
\section{Appendix: Direct sums of Banach spaces} \label{appendix direct sums}


A closer look at Definition \ref{def ell spaces} reveals that the $\ell^{(\mu, \theta, r)}$-norms are built up from two nested norms: The inner part \eqref{formula power space norm inner part} of the $(\mu,\theta,r)$-norm originates from an $\ell^2(\mu^{-\theta})$-norm taken over the restricted indices in $M_j$, while the outer part in \eqref{formula power space norm} corresponds to the $\ell^r$-norm taken over the values given by the inner part. 

To deal with these norms, let us introduce such nested norms in a more general setting. 
For notational convenience, let us assume that $I = \{1,...,n\}$ or $I = \naturals$ is an index set. 

\begin{definition} \label{def direct sum}
	Let $I$ be an index set and suppose that $E_i$ are Banach spaces for all $i \in I$. 
	Then for $1 \leq p \leq \infty$ we define the $p$-direct sum of the $E_i$ as 
	\begin{align}
		{\bigoplus_{i \in I}}^{(p)} E_i
		:= \left\{ x=(x_i)_{i \in I} \, \middle| \, x_i \in E_i \text{ and } \Vert x \Vert_p  < \infty \right\} \text{,}
	\end{align}
	where the norm is defined by $\Vert x \Vert_p := \Vert ( \Vert x_i \Vert_{E_i} )_{i \in I} \Vert_{\ell^p}$. 
	We define the $0$-direct sum as 
	\begin{align}
		{\bigoplus_{i \in I}}^{(0)} E_i
		:= \left\{ x=(x_i)_{i \in I} \, \middle| \, x_i \in E_i \text{ and } \lim_{i \rightarrow \infty} \Vert x_i \Vert_{E_i} = 0 \right\}
	\end{align}
	\anmerkung{$I=\naturals$ auch hier nicht nötig, da Grenzwert=einziger Häufungspunkt}where the limit of a finite sequence is defined as $0$. We equip this space with the norm $\Vert x \Vert_{0} := \Vert x \Vert_{\infty}$. 
\end{definition}
It is well known that such $p$-direct sums of Banach spaces $E_i$ are again Banach spaces, see e.g. \cite[Section 16.11]{aliprantis_border}. 

\begin{remark} \label{remark absoulte values in sums}
	Let $E_i$ be Banach spaces for indexes $i \in I$, let $1 \leq p \leq \infty$, $C>0$ and suppose that $x_i' \in E_i'$ for all $i \in I$. 
	Then the following holds: 
	\begin{align}
		\sum_{i \in I} x_i'(x_i) \text{ converges for all } x \in {\bigoplus_{i \in I}}^{(p)} E_i \; &\Leftrightarrow \; \sum_{i \in I} |x_i'(x_i)| \text{ converges for all } x \in {\bigoplus_{i \in I}}^{(p)} E_i \\
		\sum_{i \in I} x_i'(x_i) \leq C \Vert x \Vert_p \text{ for all }  x \in {\bigoplus_{i \in I}}^{(p)} E_i \; &\Leftrightarrow \; \sum_{i \in I} |x_i'(x_i)| \leq C \Vert x \Vert_p \text{ for all }  x \in {\bigoplus_{i \in I}}^{(p)} E_i \\
		\sup_{\Vert  x \Vert_p \leq 1} \sum_{i \in I} x_i'(x_i) \; &= \; \sup_{\Vert  x \Vert_p \leq 1} \sum_{i \in I} |x_i'(x_i)|
	\end{align}
	Thereby, \enquote{$\Leftarrow$} and \enquote{$\leq$} follow directly. To see \enquote{$\Rightarrow$} and \enquote{$\geq$}, we choose $\tilde{x}_i:= x_i$ if $x_i'(x_i) \geq 0$ and $\tilde{x}_i:= -x_i$ otherwise. Then we find that $x_i'(\tilde{x}_i) = |x_i'(x_i)|$, and we still have $\Vert \tilde{x} \Vert_p = \Vert x \Vert_p$. Hence using $\tilde{x}$ instead of $x$ on the left-hand side gives the right-hand side. 
	
	In addition, note that due to 
	\begin{align}
		\sum_{i \in I} x_i'(x_i) \leq \left| \sum_{i \in I} x_i'(x_i) \right| \leq \sum_{i \in I} | x_i'(x_i) |
	\end{align}
	we might replace one of the sums by $\left| \sum_{i \in I} x_i'(x_i) \right|$ and still obtain the equivalences. 
\end{remark}

\begin{lemma} \label{lemma dual of direct sum} \label{lemma dual of infty direct sum}
	Let $I$ be an at most countable index set and suppose that $E_i$ are Banach spaces for all $i \in I$. 
	Then the mapping 
	\begin{align} \label{formula isometry or isometric isomorphism to dual}
		{\bigoplus_{i \in I}}^{(p')} E_i' \rightarrow \left( {\bigoplus_{i \in I}}^{(p)} E_i \right)', 
		\quad 
		x' \mapsto \left( x \mapsto \sum_{i \in I} x'_i (x_i) \right)
	\end{align}
	is 
	\begin{enumerate}
		\item an isometric isomorphism if $p=0$ and $p'=1$, 
		\item an isometric isomorphism if $1 \leq p < \infty$ with $\frac{1}{p} + \frac{1}{p'} = 1$, and 
		\item an isometry if $p=\infty$ and $p'=1$. 
	\end{enumerate}
\end{lemma}	
\begin{proof} \labelirgendwo{proof lemma dual of infty direct sum}
	A proof of (i) and (ii) can be found in \cite[Theorem 16.49 and the following remark]{aliprantis_border}\anmerkung{Nach dem Theorem kommen die Fälle $p=1$ und $p=0$}.  
	Let us prove (iii) here. 
	We begin by showing that the operator is well-defined in the sense that the sum converges. In fact, we even have absolute convergence, as for  $x' \in {\bigoplus_{i \in I}}^{(1)} E_i'$ and $x \in {\bigoplus_{i \in I}}^{(\infty)} E_i$, we find that 
	\begin{align}
		\sum_{i \in I} | x'_i (x_i) |
		\leq \sum_{i \in I} \Vert x_i' \Vert_{E_i'} \Vert x_i \Vert_{E_i}
		\leq \Vert x' \Vert_{1} \Vert x \Vert_{\infty}
		< \infty \text{.}
	\end{align}
	To show that \eqref{formula isometry or isometric isomorphism to dual} is an isometry, let $x' \in {\bigoplus_{i \in I}}^{(1)} E_i'$. 
	We have to show that the operator norm of 
	\begin{align}
		{\bigoplus_{i \in I}}^{(\infty)} E_i \rightarrow \reals, 
		\quad
		x \mapsto \sum_{i \in I} x'_i (x_i)
	\end{align}
	agrees with $\Vert x' \Vert_{1}$. 
	To this end, note that
	\begin{align}
		\left\Vert x \mapsto \sum_{i \in I} x'_i (x_i) \right\Vert
		= \sup_{\Vert x \Vert_{\infty}\leq 1} \left| \sum_{i \in I} x'_i (x_i) \right| 
		&= \sup_{\substack{\Vert x_i \Vert_{E_i} \leq 1 \\ \text{for all } i \in I}}  \sum_{i \in I} x'_i (x_i)  \\
		&= \sum_{i \in I} \sup_{\Vert x_i \Vert_{E_i} \leq 1 } x'_i (x_i) \\
		&= \sum_{i \in I} \Vert x_i' \Vert \\
		&= \Vert x' \Vert_1 \text{.}
	\end{align}
	Concerning the absolute values, we used Remark \ref{remark absoulte values in sums}. 
	This shows that the mapping in question is in fact an isometry. 
\end{proof}

For $p$ and $p'$ as in Lemma \ref{lemma dual of direct sum}, we introduce for $x \in {\bigoplus_{i \in I}}^{(p)} E_i$ and $x' \in {\bigoplus_{i \in I}}^{(p')} E_i'$ the dual pairing 
\begin{align}
	\skalarprodukt{x'}{x} := \sum_{i \in I} x'_i (x_i) \text{.}
\end{align}
Note that in situations (i) and (ii), Lemma \ref{lemma dual of direct sum} allows us to describe the entire dual space. 
However, by part (iii) we can at least understand part of the dual space in a similar fashion as we can understand part of $(\ell^{\infty}(I))'$ by the embedding $\iota: \ell^1(I) \rightarrow \ell^1(I)'' \simeq (\ell^{\infty}(I))'$. 

\begin{lemma} \label{lemma finite norm by dual evaluation}
	Let $I$ be an at most countable index set and suppose that $E_i$ are Banach spaces for all $i \in I$. Let $x'_i \in E_i'$ for all $i \in I$ and suppose that $1 \leq p,p' \leq \infty$ with $\frac{1}{p}+\frac{1}{p'}=1$. 
	Then the following are equivalent: 
	\begin{enumerate}
		\item The sum 
		\begin{align} \label{formula sum over functionals}
			\sum_{i \in I} x'_i(x_i)
		\end{align}
		converges for all sequences $x \in \bigoplus_{i \in I}^{(p)} E_i$. 
		\item The sum \eqref{formula sum over functionals} converges absolutely for all sequences $x \in \bigoplus_{i \in I}^{(p)} E_i$. 
		\item We have $x' \in {\bigoplus_{i \in I}}^{(p')} E_i'$. 		
	\end{enumerate}
\end{lemma}
\begin{proof} \labelirgendwo{proof lemma finite norm by dual evaluation}
	For finite $I$, all norms are finite sums, hence the assertion follows. So let us assume $I=\naturals$ without loss of generality. 
	
	\enquote{(i) $\Leftrightarrow$ (ii)} follows by Remark \ref{remark absoulte values in sums}. 
	
	\enquote{(iii) $\Rightarrow$ (i)}. Let $x' \in {\bigoplus_{i \in I}}^{(p')} E_i'$. Then an application of Lemma \ref{lemma dual of direct sum} yields that the functional 
	\begin{align}
		\gamma: {\bigoplus_{i \in I}}^{(p)} E_i \rightarrow \reals, 
		\quad x \mapsto \sum_{i \in I} x'_i(x_i)
	\end{align}
	is well-defined with $\Vert \gamma \Vert = \Vert x' \Vert_{p'}$. 
	In particular, all sums in \eqref{formula sum over functionals} converge. 

	\enquote{(ii) $\Rightarrow$ (iii)}. 
	For $n \in \naturals$, let us consider the bounded operator 
	\begin{align}
		S_n : {\bigoplus_{i \in I}}^{(p)} E_i \rightarrow \reals, 
		\quad x \mapsto \sum_{i=1}^n x'_i (x_i) \text{.}
	\end{align}
	Notice that for all $x \in  {\bigoplus_{i \in I}}^{(p)} E_i$, it holds that 
	\begin{align}
		\sup_{n \in \naturals} \vert S_n (x) \vert
		= \sup_{n \in \naturals} \left\vert \sum_{i=1}^n x'_i (x_i) \right\vert
		\leq \sup_{n \in \naturals} \sum_{i=1}^n |x'_i (x_i)|
		\leq \sum_{i=1}^{\infty} |x'_i (x_i)|
		< \infty \text{,}
	\end{align}
	where we used the absolute convergence of \eqref{formula sum over functionals}. 
	Hence the uniform boundedness principle \cite[Theorem 4.52]{rynne_youngson} yields that $\sup_{n \in \naturals} \Vert S_n \Vert < \infty$. 
	
	Our next goal is to show 
	\begin{align}
		\Vert (x'_i \mathbbm{1}_{i \leq n})_{i \in I} \Vert_{p'} = \Vert S_n \Vert \text{.}
	\end{align}
	
	Case $1\leq p < \infty$. 
	By Lemma \ref{lemma dual of direct sum}, we can understand $S_n \in \left( \bigoplus_{i \in I}^{(p)} E_i \right)'$ as the functional associated to the sequence $(x_i \mathbbm{1}_{i \leq n})_{i \in I}$. Hence 
	$
	\Vert S_n \Vert 
	= \Vert (x'_i \mathbbm{1}_{i \leq n})_{i \in I} \Vert_{p'}
	$ as desired. 
	
	Case $p = \infty$. Now we have $S_n \in \left( \bigoplus_{i \in \naturals}^{(\infty)} E_i \right)'$. As Lemma \ref{lemma dual of direct sum} does not directly cover this dual space, we argue differently here. 
	To this end, note that for the operator norm we find the following: \anmerkung{Letztendlich nutzt man, dass $S_n$ im Bild von $\iota$ (Einbettung ins Bidual) liegt}
	\begin{align}
		\Vert S_n \Vert
		= \sup_{\substack{x \in \bigoplus_{i \in \naturals}^{(\infty)} E_i \\ x \neq 0}} \frac{|S_n x|}{\Vert x \Vert_{\infty}}
		= \sup_{\substack{x \in \bigoplus_{i \in \naturals}^{(\infty)} E_i \\ x \neq 0}} \frac{|S_n ((x_i \mathbbm{1}_{i \leq n})_{i \in \naturals}|}{\Vert x \Vert_{\infty}}
		= \sup_{\substack{x \in \bigoplus_{i \leq n}^{(0)} E_i \\ x \neq 0}} \frac{|S_n x|}{\Vert x \Vert_{\infty}} 
	\end{align}
	This means that the operator norm of $S_n$ equals the operator norm of $S_n$ restricted to ${\bigoplus_{i \in \naturals}}^{(0)} E_i$. 
	This restricted operator is an element of $({\bigoplus_{i \in \naturals}}^{(0)} E_i)' \simeq {\bigoplus_{i \in \naturals}}^{(1)} E_i'$, where we are using the isometric isomorphism from Lemma \ref{lemma dual of direct sum}.  
	This isometric isomorphism assigns $S_n |_{\bigoplus_{i \leq n}^{(0)} E_i}$ to the sequence $(x'_i \mathbbm{1}_{i \leq n})_{i \in I} \in \bigoplus_{i \leq n}^{(1)} E_i'$. 
	We hence find the equality of norms 
	\begin{align}
		\Vert S_n \Vert 
		= \left\Vert S_n |_{\bigoplus_{i \leq n}^{(0)} E_i} \right\Vert
		= \left\Vert (x'_i \mathbbm{1}_{i \leq n})_{i \in I} \right\Vert_{1} 
	\end{align}
	which finishes the second case. 
	
	Finally, we can put together the previous findings to prove that 
	\begin{align}
		\Vert x' \Vert_{p'} 
		= \left( \sum_{i \in \naturals} \Vert x'_i \Vert^{p'} \right)^{\frac{1}{p'}} 
		&= \sup_{n \in \naturals} \left( \sum_{i \leq n} \Vert x'_i \Vert^{p'} \right)^{\frac{1}{p'}} \\
		&= \sup_{n \in \naturals} \Vert (x'_i \mathbbm{1}_{i \leq n})_{i \in \naturals} \Vert_{p'} \\
		&= \sup_{n \in \naturals} \Vert S_n \Vert
		< \infty
	\end{align}
	as desired. 
\end{proof}
Lemma \ref{lemma finite norm by dual evaluation} does not provide concrete bounds on \eqref{formula sum over functionals} or on $\Vert x'\Vert_{p'}$. However, in combination with the previous results, we can derive a quantitative version directly from Lemma \ref{lemma finite norm by dual evaluation}. 

\begin{proposition}\label{proposition finite norm by dual evaluation with constants}
	Let $I$ be an at most countable index set and suppose that $E_i$ are Banach spaces for all $i \in I$. Let $x'_i \in E_i'$ for all $i \in I$ and suppose that $1 \leq p,p' \leq \infty$ with $\frac{1}{p}+\frac{1}{p'}=1$. Let $C \geq 0$. 
	Then the following are equivalent: 
	\begin{enumerate}
		\item For all sequences $x \in {\bigoplus_{i \in I}}^{(p)} E_i$, the sum $\sum_{i \in I} x'_i(x_i)$ converges and it holds that 
		\begin{align}
			\left| \sum_{i \in I} x'_i(x_i) \right| \leq C \Vert x \Vert_p \text{.}
		\end{align}
		\item It holds that $x' \in {\bigoplus_{i \in I}}^{(p')} E_i'$ with $\Vert x' \Vert_{p'} \leq C$. 
	\end{enumerate}
\end{proposition}
\begin{proof}
	At first, let $x' \in \bigoplus_{i \in \naturals}^{(p')} E_i'$.	By the isometric embeddings from Lemma \ref{lemma dual of direct sum}, the functional 
	\begin{align}
		\gamma_{x'}: {\bigoplus_{i \in \naturals}}^{(p)} E_i \rightarrow \reals, 
		\quad x \mapsto \sum_{i \in I} x_i'(x_i)
	\end{align}
	satisfies $\Vert \gamma_{x'} \Vert = \Vert x' \Vert_{p'}$. 
	
	\enquote{(i) $\Rightarrow$ (ii)}. An application of Lemma \ref{lemma finite norm by dual evaluation} (i) $\Rightarrow$ (iii) shows that $x' \in \bigoplus_{i \in \naturals}^{(p')} E_i'$. Hence using Remark \ref{remark absoulte values in sums} we find that 
	\begin{align}
		\Vert x' \Vert_{p'}
		= \Vert \gamma_{x'} \Vert
		= \sup_{\Vert x \Vert_p \leq 1} \sum_{i \in I} x_i'(x_i)
		= \sup_{\Vert x \Vert_p \leq 1} \left| \sum_{i \in I} x_i'(x_i) \right|
		\leq C \text{.}
	\end{align}
	\enquote{(ii) $\Rightarrow$ (i)}. The convergence of the sum follows by Lemma \ref{lemma finite norm by dual evaluation}. 
	For $x=0$, the desired estimate is clear. For $x \neq 0$, we observe that 
	\begin{align}
		\left| \frac{\sum_{i \in I} x_i'(x_i)}{\Vert x \Vert_p} \right|
		= \left| \sum_{i \in I} x_i'\left( \frac{x_i}{\Vert x \Vert_p} \right) \right|
		\leq \sup_{\Vert \tilde{x}  \Vert_p \leq 1} \left| \sum_{i \in I} x_i'(\tilde{x}_i)  \right|
		= \Vert \gamma_{x'} \Vert
		= \Vert x' \Vert_{p'}
		\leq C
	\end{align}
	and note that the assertion follows by treating absolute values as explained in Remark \ref{remark absoulte values in sums}. 
\end{proof}

\stepcounter{appendix}
\section{Appendix: Semicontinuous functions}

Let $(X,\mathcal{T})$ be a topological space. Then a function $g: X \rightarrow [0,\infty]$ is called lower semicontinuous (l.s.c.) if and only if the sets $g^{-1} ((a,\infty])$ are open in $X$ for all $a \geq 0$. 

\begin{lemma} \label{lemma non-negative lower semicontinuous}
	Let $(X,\mathcal{T})$ be a topological space and suppose that $g_i:X \rightarrow [0,\infty]$ are l.s.c. functions for all $i$ in an index set $I$. 
	Then the following functions $X \rightarrow [0,\infty]$ are l.s.c.: 
	\begin{enumerate}
		\item $\sup_{i \in I} g_i$ 
		\item $\sum_{i \in I} g_i$ 
		\item $g_1 \cdot g_2$ 
		\item $g_1^\alpha$ for all $\alpha>0$ 
	\end{enumerate}
\end{lemma}
\begin{proof}
	
	\ada{i} Non-negativity follows directly, for lower semicontinuity, see \cite[§6.2 Proposition 2]{Bourbaki_topology}. \anmerkung{Definition reellwertige Funktion mit $\infty$ §5.1}
	
	\ada{ii} Non-negativity follows directly. By \cite[§6.2 Proposition 2]{Bourbaki_topology}, finite sums of l.s.c. functions are l.s.c. The property follows for infinite sums as $\sum_{i \in I} g_i(x) = \sup_{j \in I} \sum_{i=1}^j g_i(x)$ holds for non-negative functions. 
	
	\ada{iii} Non-negativity follows directly, for lower semicontinuity, see \cite[§6.2 Proposition 2]{Bourbaki_topology}. 
	
	\ada{iv} This follows from the fact that for all $a \geq 0$  we have 
	\begin{align*}
		\{x \in X \, | \, g_1(x)^\alpha > a\} = \{x \in X \, | \, g_1(x) > a^{\frac 1\alpha}\} \text{,}
	\end{align*}
	which is an open set. 
\end{proof}

\begin{lemma} \label{lemma sup is esssup}
	Let $(X,\mathcal{T})$ be a topological space and $(X,\mathcal{A},\nu)$ be a measure space such that $\nu$ is $\mathcal{T}$-positive. Then for all l.s.c. functions $g: X \rightarrow [0,\infty]$ it holds that 
	\begin{align}
		\essup_{x \in X} g(x) = \sup_{x \in X} g(x) \text{.}
	\end{align} 
\end{lemma}
\begin{proof}
	Assume  the converse, that means $\essup_{x \in X} g(x) < \sup_{x \in X} g(x)$. It follows that there are $\varepsilon>0$ and $y \in X$ such that the set 
	\begin{align}
		M := \{ y \in X \, |  \, g(y) - \essup_{x \in X} g(x) \in (\varepsilon, \infty]\}
	\end{align}
	is not empty. Additionally, $M$ is open as $g$ is l.s.c. 
	By the positivity of $\nu$, we find that $\nu(M)>0$. Hence $g > \essup_{x \in X} g(x)$ holds on the set of positive measure $M$, a contradiction. 
\end{proof}

%

\begin{corollary} \label{corollary L infty = L infty}
	Let $(X,\mathcal{T})$ be a topological space and $(X,\mathcal{A},\nu)$ be a measure space such that $\nu$ is $\mathcal{T}$-positive. Let $f:X \rightarrow \reals$ be continuous. 
	Then $\Vert f \Vert_{\mathcal{L}^\infty (X)} = \Vert [f]_\sim \Vert_{\gL^\infty (\nu)}$. 
\end{corollary}
\begin{proof}
	Apply Lemma \ref{lemma sup is esssup} to the non-negative and $\mathcal{T}$-continuous function $|f|$. 
\end{proof}

\begin{lemma} \label{lemma Cb closed}
	Let $(X,\mathcal{T})$ be a topological space and $(X,\mathcal{A},\nu)$ be a measure space such that $\nu$ is $\mathcal{T}$-positive. 
	Then $\{ [f]_\sim \, | \, f \in \gC_b(X,\mathcal{T})\}$ is a closed subspace of $\gL^\infty(\nu)$ which is isometric isomorphic to $\gC_b(X,\mathcal{T})$ by the mapping $f \mapsto [f]_\sim$ for $f \in \gC_b(X,\mathcal{T})$. 
\end{lemma}
\begin{proof}
	The fact that $\{ [f]_\sim \, | \, f \in \mathrm{C}_b(X,\mathcal{T})\}$ is a linear subspace of $\gL^\infty(\nu)$ follows from the definition. 
	It remains to prove that it is closed. We do so by proving completeness, which is equivalent to closedness. 
	To this end, let $(f_n)_{n \in \naturals}$ be a sequence of functions in $\gC_b (X,\mathcal{T})$ such that the $[f_n]_\sim$ form a Cauchy sequence in $\gL^\infty(\nu)$-norm. It follows by Corollary \ref{corollary L infty = L infty} that for $n,m \rightarrow \infty$, we have 
	\begin{align}
		\Vert f_n - f_m \Vert_{\mathcal{L}^\infty (X)}
		= \Vert [f_n - f_m]_\sim \Vert_{\gL^\infty (\nu)}
		= \Vert [f_n]_\sim - [f_m]_\sim \Vert_{\gL^\infty (\nu)}
		\rightarrow 0 \text{.}
	\end{align}
	As $\gC_b(X,\mathcal{T})$ with the $\mathcal{L}^\infty (X)$-norm is complete, there is a $f \in \gC_b(X,\mathcal{T})$ such that $\Vert f_n - f \Vert_{\mathcal{L}^\infty (X)} \rightarrow 0$. But this implies 
	\begin{align}
		\Vert [f_n]_\sim - [f]_\sim \Vert_{\gL^\infty (\nu)}
		=\Vert [f_n - f]_\sim \Vert_{\gL^\infty (\nu)}
		= \Vert f_n - f \Vert_{\mathcal{L}^\infty (X)} 
		\rightarrow 0
	\end{align}
	by another application of Corollary \ref{corollary L infty = L infty}. 
	This shows that $[f]_\sim$ is the desired limit of $([f_n]_\sim)_{n \in \naturals}$ in $\gL^\infty(\nu)$-norm. 
	
	Finally, as the $\mathcal{L}^\infty (X)$-norm of a function $f \in \gC_b(X, \tau(H))$ agrees with the $\gL^\infty(\nu)$-norm of $[f]_\sim$ by Corollary \ref{corollary L infty = L infty}, the map 
	\begin{align}
		\mathrm{C}_b(X,\mathcal{T}) \rightarrow \{ [f]_\sim \, | \, f \in \mathrm{C}_b(X,\mathcal{T})\}, 
		\quad f \mapsto [f]_\sim
	\end{align}
	is isometric, hence injective. Its surjectivity follows directly. 
\end{proof}

	\bibliographystyle{plain}
	\bibliography{../bib/bib.bib}

\end{document}